\documentclass{amsart}

\usepackage{amssymb,amsmath}
\usepackage{amsthm}
\usepackage{fancyhdr}
\usepackage{tikz}
\usetikzlibrary{decorations.markings}

\newcommand{\RR}{{\mathbb R}}
\newcommand{\CC}{{\mathbb C}}

\newcommand{\HH}{\mathcal{H}}
\newcommand{\Jq}{\langle q \rangle}

\DeclareMathOperator{\sech}{sech}
\DeclareMathOperator{\Tanh}{tanh}

\newtheorem{theorem}{Theorem}
\newtheorem{corollary}{Corollary}
\newtheorem{lemma}{Lemma} 
\newtheorem{proposition}{Proposition}

\newtheorem{remark}{Remark}

\begin{document}

\title[]{Global solution of the Wadati-Konno-Ichikawa equation with small initial data}
\author{Yusuke Shimabukuro} 
\address{ Institute of Mathematics, Academia Sinica, Taipei, Taiwan, 10617}
\email{shimaby@gate.sinica.edu.tw}

\date{\today}
\maketitle

\begin{abstract}
We show the existence of global solution for the Wadati-Konno-Ichikawa (WKI) equation with small initial data in the smooth function space. Our approach is based on the scattering and inverse scattering maps in the weighted Sobolev spaces for the WKI spectral problem. In addition, we derive one soliton solution as well as a bursting soliton from our matrix Riemann-Hilbert problem.
\end{abstract}

\section{Introduction}
The paper addresses the initial value problem of the Wadati-Konno-Ichikawa (WKI) equation,
\begin{equation} \label{WKI}
\left\{ \begin{array}{l}
iq_t+\left(\frac{q}{\sqrt{1+|q|^2}}\right)_{xx}=0, \\
\left. q\right|_{t=0}=q_0, \end{array} \right. 
\end{equation}
where $q(x,t)$ is a complex-valued function on $(x,t)\in \RR\times [0,\infty)$. The WKI equation is integrable and has an infinite number of conserved quantities. The first two conserved quantities of the WKI equation are 
$$E_1=\int_{\RR}(\sqrt{1+|q|^2}-1)dx,$$
$$E_2=\int_{\RR}( \frac{1}{2}\frac{(|q|^2)_x}{1+|q|^2}+\frac{q_x}{q}\frac{1-\sqrt{1+|q|^2}}{\sqrt{1+|q|^2}})dx,$$
satisfying $\frac{d}{dt}E_j=0$, $j=1,2$.

The two WKI equations are introduced in \cite{Wadati-Konno-Ichikawa-1979}. The first equation is \eqref{WKI} and the second equation is 
\begin{equation} \label{WKI-2}
q_t+\left(\frac{q_x}{(1+|q|^2)^{3/2}}\right)_{xx}=0,
\end{equation}
which can be derived to model oscillation of elastic beams \cite{Ichikawa-Konno-Wadati-1981}. 

Shimizu and Wadati first studied the WKI equation \eqref{WKI} by the inverse scattering transformation \cite{Shimizu-Wadati-1980}. They discovered that one soliton solution with a particular parameter becomes a bursting soliton that has infinite hight locally in space. Wadati, Konno, and Ichikawa considered a modified version of \eqref{WKI-2} and obtained a loop soliton solution \cite{Konno-Ichikawa-Wadati-1981-2}.
Around the same time, the same authors showed that
$$q_t-2\left(\frac{1}{\sqrt{1+q}}\right)_{xxx}=0$$
is integrable and has a cusp soliton solution \cite{Wadati-Ichikawa-Shimizu-1980}, whose slope is infinite at some point.

The above integrable equations admit the exotic soliton solutions that are not smooth. Other well-known soliton equations of this type include the Camassa-Holm equation, the Degasperis-Procesi equation, the Harry-Dym equation, the Short-Pulse equation and so fourth. Numerous results on well-posedness and wave breaking are obtained in recent years among these equations. We only list a shortlist. See \cite{Constantin-Escher-1998} for the Camassa-Holm equation by Constantin and Escher, \cite{Liu-Yin-2006} for the Degasperis-Procesi equation by Liu and Yin, and \cite{Pelinovsky-Sakovich-2010,Liu-Pelinovsky-Sakovich-2009} for the Short-Pulse equation by Liu, Pelinovsky and Sakovich.  
What above results have in common is that the global well-posedness can be achieved with some initial profile but not every initial profile in the same function space, since there exists an initial profile such that slope of a corresponding solution becomes infinity at some point in a finite time.
 
The WKI equation admits a bursting soliton whose maximum amplitude is infinity. This type of singularity differs with peakon type solutions. 
Natural questions rise for the WKI equation what conditions are required for the global existence of solution and whether blow-up phenomena can happen or not. In this paper, we address the first question by the inverse scattering method. We also derive one soliton solution as well as a bursting soliton by the framework of our matrix Riemann-Hilbert problem that appears to be simpler than the original derivation \cite{Shimizu-Wadati-1980}. More exact soliton solutions will be reported in our follow-up work. 

Solution of the WKI equation appears in the Wadati-Konno-Ichikawa spectral problem,  
\begin{equation} \label{WKI-spectral}
\psi_x=[i\lambda \sigma_3-\lambda \mathcal{M}]\psi, \quad \mathcal{M}=\begin{pmatrix} 0 & q\\ -\bar{q} & 0 \end{pmatrix},
\end{equation} 
where $\lambda \in \CC$.
Formally, given the time evolution 
\begin{equation} \label{WKI-spectral-t}
\psi_t=[-2i\lambda^2\sigma_3-2i\frac{1-\sqrt{1+|q|^2}}{\sqrt{1+|q|^2}}\lambda^2\sigma_3+2\frac{1}{\sqrt{1+|q|^2}}\lambda^2 \mathcal{M}-i\lambda \; \mathcal{N}]\psi 
\end{equation}
where $ \mathcal{N}=\begin{pmatrix} 0 & \left(\frac{q}{\sqrt{1+|q|^2}}\right)_x\\ \left(\frac{\bar{q}}{\sqrt{1+|q|^2}}\right)_x & 0 \end{pmatrix}$, $\partial_{xt}\psi=\partial_{tx}\psi$ gives the WKI equation \eqref{WKI}.

The inverse scattering transform by Shimizu and Wadati \cite{Shimizu-Wadati-1980} is based on the Gelfand-Levitan equation under assumption of compact support for the potential. Boiti, Gerdjikov, and Pempinelli studied the squared solutions of the WKI spectral problem with the potential in the Schwartz function space. A relationship between the Ablowitz-Kaup-Newell-Segur (AKNS) problem and the WKI problem is shown in \cite{Ishimori-1982} by Ishimori. One soliton solution is derived for the WKI evolution equations by using the known solitons of the AKNS evolution equations. This relationship between the ANKS and WKI spectral problems is reported by Wadati and Sogo \cite{Wadati-Sogo-1983} as well. However, it is hard to tell whether existence of solution for a AKNS evolution equation implies that of solution for a WKI evolution equation since the latter solution contains the inverse power of a ANKS solution, e.g., in \cite{Ishimori-1982} (equation (3.6)). Furthermore, in \cite{Wadati-Sogo-1983} (equation (4.9)), the Isotropic Heisenberg equation is related to the WKI equation whose solution contains the inverse power of the former. This is not the case for relation of the Isotropic Heisenberg equation to the nonlinear Schr\"{o}dinger equation. 

We study the scattering and inverse scattering maps for the potential $q$ in weighted Sobolev spaces. 
We follow the machinery of the Riemann-Hilbert approach by Deift and Zhou \cite{Deift-Zhou-1991, Zhou-1989, Zhou-1998}. Their work on long-time asymptotics of the defocusing NLS \cite{Deift-Zhou-2002} also contains many important details of the method. 
 
 At present we are unable to close regularities of the potential $q$ in a single space under the scattering and inverse scattering maps. We, therefore, require the infinite smoothness of the potential $q$ in order to conclude the solvability of the WKI equation by the inverse scattering transform. Existence of local solution of \eqref{WKI} can be obtained in the same smooth space. We denote $L^2(\RR)$ as the space of square integrable functions equipped with the norm,
$$\|f\|_{L^2(\RR)}:=\left(\int_{\RR}|f(x)|^2dx\right)^{1/2}.$$  
We will denote $X_n$ as the function space defined by
$$X_n:=\{ f \in L^2(\RR) :
\langle x\rangle \partial_x^jf
\in L^2(\RR)   \quad \forall j=0,1,\cdots,n\},$$
where $\langle x\rangle=\sqrt{1+x^2}$, and $\partial_x$ is the weak partial derivative.
When $n=\infty$, we have
$$X_{\infty}=\{f\in L^2(\RR): \| \langle x \rangle \partial_x^j f \|_{L^2} < \infty \quad \forall j \in \mathbb{N} \}.$$
Every function in $X_{\infty}$ is smooth and any $j$ th derivative of a function with a weight $\langle x\rangle$ belongs to $L^2$. 

From Theorems \ref{b-a-regularity} and \ref{construct-q}, the potential $q \in X_{n+1}$ after a composition of the scattering and the inverse scattering maps is 
$$X_{n+1}\ni q \mapsto q\in X_{n-1}, \quad n\geq 2.$$

As Corollary, the following result is obtained.
\begin{theorem}\label{main-theorem}
If $q_0 \in X_{\infty}$ and 
$$\|\langle x \rangle q_0\|_{L^2}+\|\langle x\rangle \partial_xq_0\|_{L^2}$$
is sufficiently small, then there exists a unique solution $q(\cdot,t)\in X_{\infty}$ of the initial value problem \eqref{WKI} for all $t \in \RR^+$. 
\end{theorem}

The solution implied in Theorem \ref{main-theorem} has no soliton solution because of the smallness condition on an initial data. It is worth mentioning that Theorem \ref{main-theorem} requires infinite smoothness of an initial data, and it would be interesting to see existence of local/global solution in lower regularity for the WKI equation. 

We note that our work can be readily extended to the study of long-time asymptotic solution by using the steepest descent method. We expect that the leading order decays in the order of $1/\sqrt{t}$, since the WKI equation is formally written as the nonlinear Schrodinger type equation if $|q|$ is small. As for other evolution equations from the WKI hierarchy, the long-time asymptotic solution of the Short-Pulse equation is recently reported in \cite{Xu-2016}. 

In the scattering problem, there are mainly three important steps for technical aspects, (i) transformation of the WKI spectral problem to the ANKS spectral problem for existence, analyticity, and estimates of the fundamental solutions, (ii) the change of spectral parameter $\lambda \rightarrow -\frac{1}{\lambda}$ for correct normalization, (iii) the change of the spatial variable that depends on the dependent variable to eliminate the potential $q$ in the jump condition of the Riemann-Hilbert problem. 
This type of transformation in (iii) is called the reciprocal transformation or the hodograph transformation in many literatures, e.g., see \cite{Feng-Inoguchi-Kajiwara-Maruno-2011} for integrable discretizations of the WKI elastic beam equation \eqref{WKI-2}. 



The paper is organized as follows.

 In Section \ref{WKI-AKNS}, we transform the WKI spectral problem to the AKNS type spectral problem. From the ANKS type spectral problem, we obtain various results and relate them back to the WKI spectral problem. 
 
 In Section \ref{Scattering-coefficients}, we introduce the scattering coefficients $a(\lambda)$ and $b(\lambda)$ and give their regularity based on regularity of the potential $q\in X_n$ in Theorem \ref{b-a-regularity}. 
 
 In Section \ref{inverse-transformation}, we make the change of the spectral parameter, $\lambda \mapsto -\frac{1}{\lambda}$, and give the sectionally analytic matrix function \eqref{normalized-m} that is normalized to identity at infinity.  
 
 In Section \ref{change-space-variable}, we introduce the change of the spatial variable to address an appearance of the potential $q$ in the jump condition \eqref{jump-RHP}. The new spatial variable will be denoted as $x_{\HH}$ and the new potential as $q_{\HH}$ that depends on $x_{\HH}$. We finally give the jump condition \eqref{jump-RHP-2} for the Riemann-Hilbert problem, used in the inverse problem. 
 
In Section \ref{RHP-section}, we recover and estimate the potential $q_{\HH}$ on the variable $x_{\HH}$ in Lemma \ref{qH-construct} for every $t\geq 0.$ Finally, we estimate the potential $q$ on the original variable $x$ in Theorem \ref{construct-q} for every $t\geq 0$. 

In Section \ref{soliton}, we derive one soliton solution by using our matrix Riemann-Hilbert formulation. \\


\emph{Notations}. We define $|M|:= \left(\sum_{ij}|M_{ij}|^2\right)^{1/2}$ for a $n\times m$ matrix $M$ with an $(i,j)$ th element $M_{ij}$. If $M(z)$ is a matrix function on $z\in \Sigma$ for some contour $\Sigma\subset \CC$, then $\| M(x)\|_{X}:=\| |M(x)|\|_{X}$, where $\|\cdot\|_{X}$ is the function norm for the function space $X$, e.g., the $L^2$ space on $\Sigma$. We denote $\langle x\rangle=\sqrt{1+x^2}$. When we write $L^2(\langle x \rangle^2 dx)$, we mean the weighted $L^2$ space and its norm $\| f\|_{L^2(\langle x \rangle^2dx)}=\| \langle x \rangle f\|_{L^2}$.

We frequently use the third pauli matrix $\sigma_3=\begin{pmatrix} 1 & 0 \\ 0 & -1 \end{pmatrix}$, e.g., $a^{\sigma_3}=\begin{pmatrix} a & 0 \\ 0 & a^{-1} \end{pmatrix}$ for $a\in \CC$.

\section{Transformation to the AKNS-type system} \label{WKI-AKNS}
From now on, we adapt the Japanese bracket notation $\langle \cdot \rangle$ for the potential $q$, i.e., $\langle q \rangle:= \sqrt{1+|q|^2}$. 
We perform the standard eigen decomposition for the operator $\mathcal{L}=-i\lambda \sigma_3+\lambda \mathcal{M}$,
\begin{equation} \label{eig-decomp}
\mathcal{L}=\mathcal{G}\begin{pmatrix} i\lambda\Jq & 0\\ 0 &  -i\lambda\Jq\end{pmatrix} \mathcal{G}^{-1},
\end{equation}
where 
\begin{equation} \label{G-matrix}
\mathcal{G}=\frac{1}{\sqrt{2}(\Jq^2+\Jq)^{1/2}}\begin{pmatrix}  1+\Jq & -iq \\ -i\bar{q} & 1+\Jq \end{pmatrix}.
\end{equation}
We observe that $\mathcal{G}$ invertible since $\det(\mathcal{G})=1$. We introduce the new matrix function
\begin{equation} \label{G-transformation}
\phi_1=\mathcal{G}^{-1}\psi
\end{equation}
and expressing \eqref{WKI-spectral} in terms of $\phi_1$ gives 
\begin{equation} \label{WKI-spectral-2}
\partial_x\phi_1=[\sigma_3 i\lambda \Jq+ \mathcal{V}_1]\phi_1, \quad \mathcal{V}_1=-\mathcal{G}^{-1}\mathcal{G}_x= \begin{pmatrix} B & Q \\ -\bar{Q} & \bar{B} \end{pmatrix},
\end{equation}
where 
$$Q= \frac{-i}{4(\Jq^2+\Jq)}(q\Jq_x -q_x(1+\Jq)), \quad B= \frac{1}{4}\frac{q_x\bar{q}-q\bar{q}_x}{\Jq^2+\Jq}.$$ 
We observe that $B$ is purely imaginary. We shall apply the gauge transformation to eliminate the diagonal elements in $\mathcal{V}_1$, 
\begin{equation}\label{gauge-transformation}
\phi_2=g_{x_0}\phi_1,
\end{equation}
with 
$$g_{x_0}=\exp\left( -\sigma_3 \int_{x_0}^xB(y)dy\right).$$
The subscript $x_0$ is written to keep track of limits $x_0=+\infty$ or $x_0=-\infty$.
This gauge transformation transforms \eqref{WKI-spectral-2} into
\begin{equation} \label{WKI-spectral-3}
\partial_x\phi_2=[\sigma_3 i\lambda\Jq + \mathcal{V}_2]\phi_2, \quad \mathcal{V}_2= g_{x_0}\begin{pmatrix}0 & Q \\ -\bar{Q} & 0 \end{pmatrix}g_{x_0}^{-1}.
\end{equation}
Observe that $\mathcal{V}_2$ is an off-diagonal matrix. The spectral problem \eqref{WKI-spectral-3} is an analogue to the ANKS system. Lastly, we introduce $\varphi$ to normalize $\phi_2$ as follows
\begin{equation}\label{gauge-transformation-2}
\varphi=\phi_2 e^{-i\lambda \sigma_3p_{x_0}(x)}
\end{equation}
with
\begin{equation} \label{p}
p_{x_0}(x)=x + \int_{x_0}^x(\Jq-1)dy
\end{equation}
where we note that $\int_{\RR}(\Jq-1)dx$ is a conserved quantity. 
From \eqref{WKI-spectral-3}, the integral equation for $\varphi$ with the initial condition $\varphi=I$ at $x=x_0$ is given as
\begin{equation}\label{mu-integral}
\varphi=I + \int_{x_0}^{x}e^{i\lambda(p_{x_0}(x)-p_{x_0}(y)) \sigma_3}[\mathcal{V}_2\varphi] e^{-i\lambda (p_{x_0}(x)-p_{x_0}(y))\sigma_3} dy.
\end{equation}
For solution $\varphi$ corresponding with $x_0=\pm \infty$, we simply denote $\varphi_{\pm}$ for $x_0=\pm\infty$ above.

We easily get the following estimate on the entry of $\mathcal{V}_2$, if $q\in H^1$ and $q_x \in L^1$, then
\begin{align}\label{estimate-V}
\|Q\|_{L^1} &\leq \frac{1}{8}(\|q^2q_x\|_{L^1}+\|q_x\|_{L^1}+\|q_x\Jq\|_{L^1})
\end{align}
is bounded.
 We see that $X_1$ regularity for $q$ is sufficient to have a bounded estimate above since if $q\in X_1$, then $q \in H^1$ and $q_x\in L^1$. Furthermore, the integral that appears in $p_{x_0}$ of \eqref{p} is bounded as 
 $$\int_{\RR}(\sqrt{1+|q|^2}-1)dx \leq \int_{\RR}|q|dx,$$
 where we used the fact that $\sqrt{1+s^2}-1 \leq s$ for $s \geq 0$.\\
 
We first address existence and uniqueness of bounded solutions $\varphi_{\pm}(\cdot;\lambda)$ of \eqref{mu-integral} for every fixed $\lambda \in \RR$. This follows from the fixed-point argument. If $q\in X_1$, there exist a finite number of sub-intervals of $\RR$ such that $Q$ is sufficiently small in $L^1$ norm in each sub-interval for the argument to work. By uniqueness, gluing solutions together, we deduce a unique bounded solutions $\varphi_{\pm}(\cdot,\lambda)$ for every $\lambda \in \RR$.
\begin{proposition} \label{existence-mu}
If $q\in X_1$, then there exists unique $L^{\infty}_x$ solution $\varphi_{\pm}$ in \eqref{mu-integral} for every $\lambda \in \RR$.
\end{proposition}
Furthermore, we have analyticity properties.
\begin{proposition} \label{analyticity-mu}
Let $q\in X_1$. The first column of $\varphi_+$ and the second column of $\varphi_-$ are analytic in $\CC^+$. The first column of $\varphi_-$ and the second column of $\varphi_+$ are analytic in $\CC^{-}$. They are normalized at infinity, i.e.,
$$[(\varphi_+)_1, (\varphi_-)_2] \rightarrow I, \quad [(\varphi_-)_1, (\varphi_+)_2] \rightarrow I \quad   |\lambda| \rightarrow \infty$$
for $\lambda$ in their analytic domains. 
\end{proposition}
In order to prove Proposition \ref{analyticity-mu}, it is sufficient to check the sign of $p_{x_0}(x)-p_{x_0}(y)$, i.e., 
\begin{equation} \label{exponent}
p_{x_0}(x)-p_{x_0}(y)=(x-y)+\int_{y}^x(\Jq-1)d\tau
\end{equation}
and find that $p_{x_0}(x)-p_{x_0}(y)$ is positive when $y<x$ and negative when $y>x$, i.e., the sign of the integral $\int_{y}^x(\Jq-1)d\tau$ is consistent with that of $(x-y)$. The full statement of Proposition can be proved analogously as in (\cite{Ablowitz-Prinari-Trubatch-2004}, Lemma 2.1).

In the following, we give useful estimates for the oscillatory integral involving with \eqref{exponent} in the exponent, i.e., for $h\in L^2$,
$$\int_x^{\infty}e^{2i\lambda \ell(x,y)}h(y)dy$$
where $\ell(x,y)=y-x+\int_x^y(\Jq-1)d\tau$. 

\begin{proposition} \label{Fourier-estimate}
Let $q\in X_1$. If $h \in L^2$, then 
$$\|\int_{\pm\infty}^xe^{2i\lambda \ell(x,y)}h(y)dy\|_{L^2_{\lambda}} \leq \sqrt{\pi} \|h\|_{L^2(\RR)}$$
and, furthermore, if $\langle x \rangle h \in L^2$, then
$$\|\langle x\rangle \int_{\pm\infty}^xe^{2i\lambda \ell(x,y)}h(y)dy\|_{L^2_{\lambda}} \leq \sqrt{\pi} \|\langle y \rangle h\|_{L^2(\RR)} \quad \pm x \geq 0,$$
where $\ell(x,y)=y-x+\int_x^y(\langle q(\tau)\rangle-1)d\tau.$
\end{proposition}
\begin{proof}
We shall use the change of variable $\xi=\ell(x,y)$. 
For any fixed $x\in \RR$, we denote $y(\xi)$ as a function on $\xi$ that is determined from 
$$\xi =y-x +\int_x^{y}(\Jq-1)d\tau.$$
Since $\frac{d\xi}{dy}=\Jq$ and $\xi=0$ if and only if $y=x$, the change of variable gives 
$$\int_x^{\infty}e^{2i\lambda \ell(x,y)}h(y)dy=\int_0^{\infty}e^{2i\lambda \xi}\frac{h(y(\xi))}{\sqrt{1+|q(y(\xi))|^2}}d\xi.$$
By the Plancherel's theorem, we find
\begin{align} \label{Fourier-est-1}
\left\|\int_0^{\infty}e^{2i\lambda\xi} \frac{h(y(\xi))}{\sqrt{1+|q(y(\xi))|^2}}d\xi\right\|^2_{L^2_{\lambda}} &=\pi \int_0^{\infty}\frac{|h(y(\xi))|^2}{1+|q(y(\xi))|^2}d\xi \nonumber \\
 &=\pi \int_x^{\infty}\frac{|h(y)|^2}{\sqrt{1+|q(y)|^2}}dy.
\end{align}
The first estimate in the statement follows from \eqref{Fourier-est-1}. 

The second estimate follows from estimating \eqref{Fourier-est-1} as
$$\int_x^{\infty}\frac{|h(y)|^2}{\sqrt{1+|q(y)|^2}}dy \leq \langle x\rangle^{-2} \int_x^{\infty}\frac{\langle x\rangle^2}{\langle y\rangle^2} \frac{|\langle y\rangle h(y)|^2}{\sqrt{1+|q(y)|^2}}dy\quad x\geq0.$$
The analogous way works for the integral $\int_{-\infty}^x$ as well. 
\end{proof}
We shall give various estimates involving the integral equation \eqref{mu-integral} which is expressed as
\begin{equation}\label{compact-mu-integral}
(I-K_{\pm\infty})\varphi_{\pm}=I
\end{equation}
with 
$$ 
K_{\pm \infty}\varphi_{\pm} =  \int_{\pm \infty}^{x}e^{-i\lambda\ell(x,y) \sigma_3}[\mathcal{V}_2\varphi] e^{i\lambda \ell(x,y)\sigma_3} dy.
$$
To estimate $\varphi_{\pm}-I$, we write
\begin{equation} \label{estimate-integral-1}
(I-K_{\pm\infty})(\varphi_{\pm}-I)=K_{\pm\infty}I.
\end{equation}
We recall 
$$Q=\frac{-i}{4(1+|q|^2+\sqrt{1+|q|^2})}(q(\sqrt{1+|q|^2})_x -q_x(1+\sqrt{1+|q|^2})),$$
which appears in $\mathcal{V}_2=\begin{pmatrix} 0 & Q \\ - \bar{Q} & 0 \end{pmatrix}$.
We can easily check that for $n\geq 0$
$$q \in H^{n+1} \Rightarrow Q \in H^n,$$
$$q\in X_{n+1} \Rightarrow Q\in X_n.$$
Above, we denote $H^n$ as the Sobolev space defined by
$$H^n:=\{f\in L^2(\RR): \partial_x^jf\in L^2(\RR) \quad \forall j=0,1,\cdots,n\}.$$

The following result tells that smoothness of $q$ implies weighted $L^2$ property of of $\varphi_{\pm}$ in the $\lambda$ variable.

\begin{lemma} \label{mu-weight} 
If $q\in X_{1}\cap H^{n+1}$ for $n\geq 0$, there exist $\lambda$-independent matrices $\{\mathcal{C}_j(x)\}_{j=0}^n$ with $\mathcal{C}_0=I$ and $\mathcal{C}_j(x) \in H^{n-j}(dx)$ such that 
$$\|\lambda^k(\varphi_{\pm}-\sum_{j=0}^k\frac{1}{\lambda^j}g_{\pm\infty}\mathcal{C}_j(x)g_{\pm\infty}^{-1})\|_{L^2_{\lambda}}<M$$
for all $k=0,1,\cdots, n$, where $M$ is a positive constant, independent of $x$.
\end{lemma}
 
 \begin{proof}
 The integral operator $(I-K_{\pm\infty})^{-1}$ is a bounded operator from $L^2_{\lambda}$ to $L^2_{\lambda}$ for every $x\in \RR$ if $Q \in L^1$ that holds if $q\in X_1$, i.e., we have the estimate
 \begin{equation} \label{K-inverse}
 \|(I-K_{\pm \infty})^{-1}\|_{L_x^{\infty}L_{\lambda}^2\rightarrow L_x^{\infty}L_{\lambda}^2}\leq e^{2\|Q\|_{L^1}}, 
 \end{equation} 
 where $f \in L_x^{\infty}L_{\lambda}^2$ means $\| \|f\|_{L^{2}_{\lambda}}\|_{L^{\infty}_x}<\infty.$

Also, if $Q \in L^2$, from Proposition \ref{Fourier-estimate}, we have that $K_{\pm \infty}I \in L^2_{\lambda}$. Therefore, from \eqref{estimate-integral-1}, we easily deduce that $\varphi_{\pm}-I \in L^2_{\lambda}$
for every $x\in \RR$ if $q\in X_1$. 

 By the integration by parts in the right-hand side of \eqref{estimate-integral-1}, we have 
\begin{align*}
K_{\pm\infty}I &=\int_{\pm\infty}^xe^{-i\lambda \ell(x,y) \sigma_3}g_{\pm \infty} \begin{pmatrix} 0 & Q \\ -\bar{Q} & 0 \end{pmatrix} g_{\pm \infty}^{-1}e^{i\lambda \ell(x,y)\sigma_3}dy\\
&=\frac{1}{2i\lambda \Jq}g_{\pm\infty}\mathcal{M}(Q)\sigma_3g_{\pm\infty}^{-1}+\frac{1}{\lambda}\int_{\pm \infty}^xe^{-i\lambda \ell(x,y) \sigma_3}g_{\pm\infty} \mathcal{\widetilde{M}} g_{\pm \infty}^{-1}e^{i\lambda \ell(x,y) \sigma_3}dy
\end{align*}
where we denote $\mathcal{M}(Q)=\begin{pmatrix} 0 & Q \\ -\bar{Q} & 0 \end{pmatrix}$, and $\widetilde{M}$ is an off-diagonal matrix independent of $\lambda$. It is easy to check that $\widetilde{\mathcal{M}} \in L^2(dx)$ if $q \in H^2$.  


 We can write \eqref{estimate-integral-1} as 
 $$(I-K_{\pm\infty})(\varphi_{\pm}-I-\frac{1}{2i\lambda \Jq}g_{\pm\infty}\mathcal{M}(Q)\sigma_3g_{\pm\infty}^{-1})=\frac{1}{2i\lambda}K_{\pm\infty}(g_{\pm\infty}\Jq^{-1}\mathcal{M}(Q)\sigma_3g_{\pm\infty}^{-1})$$
 $$\quad \quad +\frac{1}{\lambda}\int_{\pm \infty}^xe^{-i\lambda \ell(x,y) \sigma_3}g_{\pm\infty} \mathcal{\widetilde{M}} g_{\pm \infty}^{-1}e^{i\lambda \ell(x,y) \sigma_3}dy.$$
 
Carrying out integration by parts again and doing the same procedure, we get   
$$ 
(I-K_{\pm\infty})(\varphi_{\pm}-I-\sum_{j=2}^k\frac{1}{\lambda^j}g_{\pm\infty}\mathcal{D}_jg_{\pm\infty}^{-1}-\sum_{j=1}^k\frac{1}{\lambda^j}g_{\pm\infty}\mathcal{F}_jg_{\pm\infty}^{-1}) =
$$
\begin{equation} \label{mu-integral-weight}
\quad \quad +\frac{1}{\lambda^k}\int_{\pm \infty}^xe^{-i\lambda \ell(x,y) \sigma_3}g_{\pm\infty} \mathcal{A}_k g_{\pm \infty}^{-1}e^{i\lambda \ell(x,y) \sigma_3}dy,
\end{equation}
where $\mathcal{D}_j$ is a diagonal matrix and $\mathcal{F}_j$ is an off-diagonal matrix, and $\mathcal{D}_j, \mathcal{F}_j$, $\mathcal{A}_k$ depend only on $x$. Also, $\mathcal{D}_j$ and $\mathcal{F}_j$ contain at most $j$ th derivative of $Q$, and $\mathcal{A}_k$ contains at most $k$ th derivative of $Q$.  Since $\mathcal{A}_k\in L^2$ if $q \in H^{k+1}$, then from Proposition \ref{Fourier-estimate} the right hand-side of \eqref{mu-integral-weight} belongs to $L^2$. Since we need $q\in X_1$ to invert the operator $I-K_{\pm\infty}$ in \eqref{K-inverse}, multiplying $\lambda^k$ to \eqref{mu-integral-weight}, we need $q\in X_1\cap H^{k+1}$ to obtain the desired result.
\end{proof}

\begin{lemma} \label{mu-weight-L2} 
If $q\in X_{n+1}$ for $n\geq 0$, there exist $\lambda$-independent matrices $\{\mathcal{C}_j(x)\}_{j=0}^n$ with $\mathcal{C}_0=I$ and $\mathcal{C}_j \in H^{n-j}(dx)$ such that 
$$\|\|\lambda^k(\varphi_{\pm}-\sum_{j=0}^k\frac{1}{\lambda^j}g_{\pm\infty}\mathcal{C}_j(x)g_{\pm\infty}^{-1})\|_{L^2_{\lambda}}\|_{L^2_x(\RR_{\pm})}<M$$
for all $k=0,1,\cdots, n$, where $M$ is a positive constant.
\end{lemma}
\begin{proof}
The same argument from the above proof applies here. The only difference is that, since $q\in X_{k+1}$, then $\mathcal{A}_k$ in \eqref{mu-integral-weight} belongs to $L^2(\langle x\rangle^2dx)$. From Proposition \ref{Fourier-estimate}, \eqref{K-inverse}, and \eqref{mu-integral-weight}, we obtain
$$\|\lambda^k(\varphi_{\pm}-\sum_{j=0}^k\frac{1}{\lambda^j}g_{\pm\infty}\mathcal{C}_j(x)g_{\pm\infty}^{-1})\|_{L^2_{\lambda}} \leq M (1+x^2)^{-1/2}, \quad \pm x\geq 0,$$
where $M$ is some positive constant, independent of $x$. This proves the statement. 
\end{proof}

The following result tells that weighted $L^2$ property of $q$ implies smoothness property of $\varphi_{\pm}$ in the $\lambda $ variable.

\begin{lemma} \label{H1-mu-weight}
If $q\in X_{n+1}$ for $n\geq 0$, there exist $\lambda$-independent matrices $\{\mathcal{C}_j(x)\}_{j=0}^n$ with $\mathcal{C}_0=I$ and $\mathcal{C}_j \in H^{n-j}(dx)$ such that  
$$\|\partial_{\lambda}\left\{\lambda^k (\varphi_{\pm}(0;\lambda)-\sum_{j=0}^k\frac{1}{\lambda^j}g_{\pm\infty}(0)\mathcal{C}_j(0)g_{\pm\infty}^{-1}(0))\right\}\|_{L^2_{\lambda}} < M,$$
for all $k=0,1,\cdots,n$, where a positive constant $M>0$
\end{lemma}

\begin{proof}
Differentiating \eqref{estimate-integral-1} in $\lambda$ and setting $x=0$, we find that 
\begin{equation} \label{mu-integral-lambda}
(I-\left.K_{\pm\infty}\right|_{x=0})\partial_{\lambda}\varphi_{\pm}=(\partial_{\lambda}\left.K_{\pm \infty}\right|_{x=0})I+(\left.\partial_{\lambda}K_{\pm\infty}\right|_{x=0})(\varphi_{\pm}-I).
\end{equation}
The first term $(\partial_{\lambda}\left.K_{\pm \infty}\right|_{x=0})I$ is in $L^2_{\lambda}$ by Proposition \ref{Fourier-estimate} since $x Q \in L^2$ if $q \in X_1$, i.e.,
\begin{equation}\label{ineq1}
\|(\partial_{\lambda}\left.K_{\pm \infty}\right|_{x=0})I\|_{L^2_{\lambda}}\lesssim \|\langle y\rangle q\|_{L^2}+\|q\|_{L^2}\int_0^{\infty}(\Jq-1)dy,
\end{equation}
where $\lesssim$ hides some absolute constant.

The second term in \eqref{mu-integral-lambda} is estimated by the Minkowski inequality, $x Q \in L^2$, and the boundedness $\|\|I-\varphi_{\pm}\|_{L^{2}_{\lambda}}\|_{L^2_x(\RR_{\pm})}$ from Lemma \ref{mu-weight-L2}, i.e.,
\begin{equation} \label{ineq2}
\|  (\partial_{\lambda}K_{\pm\infty})(\varphi_{\pm}-I)\|_{L^2_{\lambda}} \lesssim  \left\{  \|\langle y\rangle q\|_{L^2}+\|q\|_{L^2}\int_0^{\infty}(\Jq-1)dy\right\}\|\|\varphi_{\pm}-I\|_{L^2_{\lambda}}\|_{L^2_x(\RR_{\pm})}.
\end{equation}
The right-hand side above is bounded if $q\in X_1$. From \eqref{mu-integral-lambda} and two previous inequalities, we deduce that if $q\in X_1$, then $\|\partial_{\lambda}\varphi_{\pm}(0;\lambda)\|_{L^2_{\lambda}} <M$, where $M$ is some constant.

If $q \in X_{k+1}$, then $\mathcal{A}_k$ in equation \eqref{mu-integral-weight} belongs to $L^2(\langle x\rangle^2dx)$, so $\| x \mathcal{A}_k\|_{L^2_x}$ is bounded. Differentiating \eqref{mu-integral-weight} in $\lambda$, multiplying $\lambda^k$, and setting $x=0$, we have 
$$(I-\left.K_{\pm\infty}\right|_{x=0})\partial_{\lambda}\left\{\lambda^k (\varphi_{\pm}-\sum_{j=0}^k\frac{1}{\lambda^j}g_{\pm\infty}\mathcal{C}_jg_{\pm\infty}^{-1})\right\}$$ 
\begin{equation}\label{eq1}
=(\left.\partial_{\lambda}K_{\pm\infty}\right|_{x=0})g_{\pm\infty}\mathcal{A}_kg_{\pm\infty}^{-1}+(\left.\partial_{\lambda}K_{\pm\infty}\right|_{x=0})\lambda^k (\varphi_{\pm}-\sum_{j=0}^k\frac{1}{\lambda^j}g_{\pm\infty}\mathcal{C}_jg_{\pm\infty}^{-1}).
\end{equation}
From this equation, $\mathcal{A}_k\in L^2(\langle x\rangle^2dx)$, and Lemma \ref{mu-weight-L2}, the first term and the second term in the right-hand side of \eqref{eq1} are estimated in the same way as \eqref{ineq1} and \eqref{ineq2}. 
\end{proof}
We have made a series of transformations \eqref{G-transformation}, \eqref{gauge-transformation}, and \eqref{gauge-transformation-2}. Solution of the WKI spectral problem \eqref{WKI-spectral} is related to $\varphi$ as follows,
\begin{equation}\label{psi-mu}
\psi_{\pm}e^{-i\lambda \sigma_3x}=\mathcal{G}g_{\pm \infty}^{-1} \varphi_{\pm}e^{+i\lambda\sigma_3\int_{\pm\infty}^x(\sqrt{1+|q|^2}-1)dy},
\end{equation}
where $\mathcal{G}$ and $g_{\pm\infty}$ are independent of $\lambda$.  We shall denote $m^{(\pm)}=\psi_{\pm}e^{-i\lambda \sigma_3x}$. We have obtained various results on $\varphi_{\pm}$ which can be used for $m^{(\pm)}$, which satisfy the following integral equation
\begin{equation}\label{m-integral-1}
m^{(\pm)}=I-\lambda\int_{\pm\infty}^x e^{i\lambda(x-y)\sigma_3}\mathcal{M}m^{(\pm)}e^{-i\lambda(x-y)\sigma_3}dy.
\end{equation}
From Proposition \ref{existence-mu} and \eqref{psi-mu}, we deduce the following result.
\begin{proposition}
Let $q\in X_1$. Then, there exists unique $L^{\infty}_x$ solution $m^{(\pm)}$ of \eqref{m-integral-1} for every $\lambda \in \RR$.
\end{proposition}
Thanks to analyticity result of $\varphi_{\pm}$ as well as their asymptotics in Proposition \ref{analyticity-mu} and \eqref{psi-mu}, we deduce the following.
\begin{proposition} \label{m-limits}
Let $q\in X_1$. Then $m^{(\pm)}$ share the same analyticity as $\varphi_{\pm}$ but not the asymptotic behaviors as $|\lambda|\rightarrow \infty$, that is, 
$$[m^{(+)}]_1e^{-i\lambda \int_{\infty}^x(\Jq-1)dy} \rightarrow \frac{1}{\sqrt{2}(\Jq^2+\Jq)^{1/2}}\begin{pmatrix} 1+\Jq \\ -i\bar{q} \end{pmatrix}e^{\int_{\infty}^xBdy}$$
$$[m^{(-)}]_2e^{i\lambda \int_{-\infty}^x(\Jq-1)dy} \rightarrow \frac{1}{\sqrt{2}(\Jq^2+\Jq)^{1/2}}\begin{pmatrix} - iq\\ 1+\Jq  \end{pmatrix}e^{-\int_{-\infty}^xBdy}$$
as $|\lambda|\rightarrow \infty$ in $\CC^+$, and 
$$[m^{(-)}]_1e^{-i\lambda \int_{-\infty}^x(\Jq-1)dy} \rightarrow \frac{1}{\sqrt{2}(\Jq^2+\Jq)^{1/2}}\begin{pmatrix} 1+\Jq \\ -i\bar{q} \end{pmatrix}e^{\int_{-\infty}^xBdy}$$
$$[m^{(+)}]_2e^{i\lambda \int_{\infty}^x(\Jq-1)dy} \rightarrow \frac{1}{\sqrt{2}(\Jq^2+\Jq)^{1/2}}\begin{pmatrix} - iq\\ 1+\Jq  \end{pmatrix}e^{-\int_{\infty}^xBdy}$$
as $|\lambda|\rightarrow \infty$ in $\CC^-$.
\end{proposition}
At this moment, we understood analyticity properties of solution $m^{(\pm)}$ of \eqref{m-integral-1}. As evident in Proposition \ref{m-limits}, however, $m^{(\pm)}$ is not normalized at $\lambda =\infty$. This issue will be addressed later.
\section{Scattering coefficient} \label{Scattering-coefficients}
We first derive three alternative forms of the scattering coefficients. From the relation \eqref{psi-mu}, apparently 
$$\psi_{\pm}e^{-i\lambda \sigma_3x}\rightarrow I \quad x\rightarrow \pm \infty.$$
We relate $\psi_+$ with $\psi_-$ through the $2\times 2$ matrix $T$,
\begin{equation} \label{def-T}
\psi_+=\psi_-T, \quad T=\begin{pmatrix} a & d \\ b & c \end{pmatrix}.
\end{equation}
Taking $x\rightarrow -\infty$ above yields 
\begin{align} \label{T-matrix-explicit}
T = I + \lambda \int_{\RR}\begin{pmatrix} -qm_{21}^{(+)} & -e^{-2i\lambda y} q m_{22}^{(+)} \\ e^{2i\lambda y} \bar{q} m_{11}^{(+)} & \bar{q}m_{12}^{(+)}\end{pmatrix} dy.
\end{align}
Using the symmetry of the spectral problem \eqref{WKI-spectral}, i.e., 
$$\psi=\begin{pmatrix} \psi_1(\lambda) \\ \psi_2(\lambda)\end{pmatrix}=\begin{pmatrix} -\overline{\psi_2(\bar{\lambda})} \\ \overline{\psi_1(\bar{\lambda})}\end{pmatrix},$$
with \eqref{T-matrix-explicit}, we verify that
\begin{equation} \label{T-matrix-a-b}
T =\begin{pmatrix} a(\lambda) & -\overline{b(\bar{\lambda})} \\ b(\lambda) & \overline{a(\bar{\lambda})} \end{pmatrix}. 
\end{equation}
\eqref{T-matrix-explicit} is an integral form for $a$ and $b$ in terms of $m^{(\pm)}$, which was conveniently used to obtain relations of scattering coefficients in \eqref{T-matrix-a-b}.

It is also useful to introduce another form of $T$ in terms of $\varphi_{\pm}$. From \eqref{psi-mu}, $\psi_+=\psi_-T$ can be written as
$$g_{+ \infty}^{-1}\varphi_+=g_{- \infty}^{-1}\varphi_- e^{i\lambda x\sigma_x+i\lambda \sigma_3 \int_{-\infty}^x \mathcal{H}dy}Te^{-i\lambda x\sigma_x-i\lambda \sigma_3 \int_{\infty}^x \mathcal{H}dy}.$$
where we used notation $\mathcal{H}=\sqrt{1+|q|^2}-1$.
From this, by taking the limit $x\rightarrow -\infty$, we find that 
\begin{equation} \label{a-explicit-mu}
a(\lambda)e^{i\lambda \int_{\RR} \mathcal{H}dy+\int_{\RR}Bdy}=1-\int_{\RR}e^{2\int_y^{\infty}Bd\tau}Q (\varphi_+)_{21}dy,
\end{equation}
where $(\varphi_{\pm})_{ij}$ is the $ij$ th element of $\varphi_{\pm}$.
Since $(\varphi_+)_{21}$ is analytic in $\CC^+$ and $(\varphi_+)_{21}\rightarrow 0$ as $|\lambda|\rightarrow \infty$ as in Proposition \ref{analyticity-mu}, from \eqref{a-explicit-mu} we have
\begin{proposition} \label{a-limits}
If $q\in X_1$, $a(\lambda)$ is analytic in $\CC^+$ and $\overline{a(\bar{\lambda})}$ is analytic in $\CC^-$ with limits,
$$a(\lambda)e^{i\lambda \int_{\RR} \mathcal{H}dy} \rightarrow e^{- \int_{\RR}Bdy}$$
$$ \overline{a(\bar{\lambda})} e^{-i\lambda \int_{\RR} \mathcal{H}dy}\rightarrow e^{ \int_{\RR}Bdy}$$
as $\lambda \rightarrow \infty$ in their analytic domains. 
\end{proposition} 

Lastly, we also give formulas of $a$ and $b$ where we have freedom to fix any value of $x\in \RR$, i.e., 
\begin{align}
a(\lambda) & =\det((\psi_+)_1,(\psi_-)_2) \nonumber\\
&=\det((\mathcal{G}g_+^{-1}\varphi_+)_1, (\mathcal{G}g_{-}^{-1}\varphi_-)_2)e^{-i\lambda \int_{\RR}\mathcal{H}dx} \label{a-det} \\
b(\lambda) &=\det((\psi_-)_1,(\psi_+)_1) \nonumber\\
&= \det(e^{2i\lambda x} (\mathcal{G}g_-^{-1}\varphi_-)_1e^{-i\lambda \int_{-\infty}^x\mathcal{H}dy}, (\mathcal{G}g_+^{-1}\varphi_+)_1e^{-i\lambda \int_{\infty}^x\mathcal{H}dy}), \label{b-det}
\end{align}
where we denoted $g_{\pm}:=g_{\pm\infty}$ for convenience and we used the notation $(*)_j$ to denote the $j$ th column of the matrix $*$.

We further rewrite \eqref{b-det} into a very useful form for estimates. Denote $\mathcal{R}_k=\sum_{j=0}^{k}\frac{1}{\lambda^j}\mathcal{C}_j(0)$ in Proposition \ref{mu-weight}. Setting $x=0$ in \eqref{b-det}, we can express it as
\begin{align} \label{b-estimate-form} 
b(\lambda)e^{i(\int_{-\infty}^0-\int_{0}^{\infty})\mathcal{H}dy}&=\det((\mathcal{G}g_-^{-1}(\varphi_-(0;\lambda)-g_-\mathcal{R}_kg_-^{-1})_1,(\mathcal{G}g_+^{-1}\varphi_+(0;\lambda))_1)+\nonumber \\
&\quad \quad +\det((\mathcal{G}\mathcal{R}_kg_+)_1, (\mathcal{G}g_+^{-1}(\varphi_+(0;\lambda)-g_+\mathcal{R}_kg_+^{-1}))_1) 
\end{align}
where we have used the very important fact that
$$\det((\mathcal{G}g_-^{-1}g_-\mathcal{R}_kg_-^{-1})_1,(\mathcal{G}g_+^{-1}g_+\mathcal{R}_kg_+^{-1})_1)=\det((\mathcal{G}\mathcal{R}_kg_-^{-1})_1,(\mathcal{G}\mathcal{R}_kg_+^{-1})_1)=0$$
since vectors are identical up to multiple and thus linearly dependent. This is why we kept track of $g_{\pm \infty}$.

We shall begin with giving regularity of $b(\lambda)$.
\begin{lemma} \label{b-estimate-2}
If $q\in X_{n+1}$, then $\lambda^jb \in L^2$ and $\partial_{\lambda}(\lambda^jb) \in L^2$ for $j=0,1, \cdots, n$.
\end{lemma}
\begin{proof}
From \eqref{b-estimate-form} when $k=0$, $b$ is estimated as 
$$\|b\|_{L^2}\leq K\|\varphi_+(0;\lambda)\|_{L^{\infty}}\|\varphi_-(0;\lambda)-I\|_{L^2}+K\|\varphi_+(0;\lambda)-I\|_{L^2},$$
where $K$ is some finite positive constant. 
The left-hand side above is bounded if $q\in X_1$ from Lemmas \ref{mu-weight} and \ref{H1-mu-weight} in the case of $n=0$. Setting $x=0$ in \eqref{b-det} and differentiating in $\lambda$, we have 
$$\partial_{\lambda}b=cb(\lambda)+\det((\mathcal{G}g_{-\infty}^{-1}\partial_{\lambda}\varphi_-(0;\lambda))_1, (\mathcal{G}g_{+\infty}^{-1}\varphi_+(0;\lambda))_1)$$
$$+\det((\mathcal{G}g_{-\infty}^{-1}\varphi_-(0;\lambda))_1, (\mathcal{G}g_{+\infty}^{-1}\partial_{\lambda}\varphi_+(0;\lambda))_1)$$
where $c=-\lambda i(\int_{-\infty}^0-\int_{0}^{\infty})\mathcal{H}dy$. The above is estimated analogously if $q\in X_1$, i.e., 
$$\|\partial_{\lambda}b\|_{L^2}\leq K\|b\|_{L^2}+K\|\varphi_+(0;\lambda)\|_{L^{\infty}}\|\partial_{\lambda}\varphi_-(0;\lambda)\|_{L^2}$$
$$\quad \quad+K\|\varphi_-(0;\lambda)\|_{L^{\infty}}\|\partial_{\lambda}\varphi_+(0;\lambda)\|_{L^2},$$
where $K$ is some positive constant. The right-hand side is bounded if $q\in X_1$. 

Since $b\in H^1$, in order to show that $\lambda^jb\in L^2$ and $\partial_{\lambda}(\lambda^jb) \in L^2$ for $j=1,2,\cdots, n$, we need to care only the behavior at infinity. Let $I=\RR\setminus U$, where $U$ is a neighborhood of the origin. Multiplying $\lambda^k$ to \eqref{b-estimate-form}, since the singularity of $\mathcal{R}_k$ at $\lambda=0$ is bounded away from $I$, we get the following estimate
$$\|\lambda^kb\|_{L^2(I)}\leq K\|\lambda^k(\varphi_-(0;\lambda)-g_-\mathcal{R}_kg_-^{-1})\|_{L^2(I)}\|\varphi_+\|_{L^{\infty}}$$
$$\quad+K\|\lambda^k(\varphi_+(0;\lambda)-g_+\mathcal{R}_kg_+^{-1})\|_{L^2(I)}$$
for some positive constant $K$, and the right-hand side is bounded for $k=0,1,\cdots,n$ if $q\in X_{1}\cap H^{n+1}$ from Lemma \ref{mu-weight}. Next, analogously, differentiating \eqref{b-estimate-form} in $\lambda$, multiplying by $\lambda^k$, we can estimate $\partial_{\lambda}(\lambda^jb)$ in $L^2(I)$ for $j=0,1,\cdots,n$ if $q\in X_{n+1}$ from Lemmas  \ref{mu-weight} and \ref{H1-mu-weight}. 
\end{proof}

Next, we turn to estimates on $a$. We first notice that \eqref{a-det} can be further expressed as 
$$
a(\lambda)e^{i\lambda\int_{\RR}\HH dy}=\det((\mathcal{G}g_+^{-1}(\varphi_+-I))_1,(\mathcal{G}g_-^{-1}\varphi_-)_2) \quad \quad \quad
$$
\begin{equation} \label{a-estimate-form}
\quad \quad\quad \quad\quad \quad+\det((\mathcal{G}g_+^{-1})_1,(\mathcal{G}g_-^{-1}(\varphi_--I))_2)+e^{-\int_{\RR}Bdy},
\end{equation}
where we have used that
$$\det((\mathcal{G}g_+^{-1})_1,(\mathcal{G}g_-^{-1})_2)=e^{-\int_{\RR}Bdy}.$$
The above equality follows from direction computations.
From Lemmas  \ref{mu-weight} and \ref{H1-mu-weight} and equation \eqref{a-estimate-form}, we deduce the following result.
\begin{lemma} \label{a-H1}
If $q \in X_1$, then $a(\lambda)e^{i\lambda \int_{\RR}\mathcal{H}dx+\int_{\RR}Bdy}-1\in L^2$ and $\partial_{\lambda}(a(\lambda)e^{i\lambda \int_{\RR}\mathcal{H}dx})\in L^2$.
\end{lemma}
Thanks to Lemma \ref{a-H1}, we see that $a(\lambda)$ is continuous if $q\in X_1$. We will show that, in the following, $a(\lambda)$ has no zero under smallness condition of $q\in X_1$.
From \eqref{a-explicit-mu}, by the reverse triangular inequality, we have
\begin{equation} \label{a-explicit-mu-estimate}
|a| \geq |1-\left|\int_{\RR}Q (\varphi_+)_{21}dy\right||.
\end{equation}
If $q\in X_1$, then Proposition \ref{existence-mu} states the unique existence of bounded solution $\varphi_{\pm}$ of the integral equation \eqref{m-integral-1}. Now, if $q \in X_1$ is sufficiently small, then one can easily verify the following. 
\begin{proposition} \label{small-solution}
For any $\delta_1>0$, there exists a $\delta_2>0$ such that if $\|q\|_{X_1}<\delta_2$, then $\|(\varphi_{+})_1-e_1\|_{L^{\infty}}<\delta_1$ for every $\lambda \in \RR\cup\CC^+$.
\end{proposition}
We denote $(\varphi_+)_1$ for the first column of $\varphi_+$.
From \eqref{a-explicit-mu-estimate} and \eqref{small-solution}, we deduce the following.
\begin{lemma} \label{a-nozero}
If $\|q\|_{X_1}$ is sufficient small, then $a(\lambda)$ is bounded away from zero for all $\lambda \in \RR\cup\CC^+$.
\end{lemma}

We also add another important property that may be apparent from \eqref{T-matrix-explicit}. We shall show it rigorously.
\begin{lemma} \label{continuity-at-zero}
$\lambda^{-1}(a(\lambda)-1)$ and $\lambda^{-1}b(\lambda)$ are continuous at $\lambda=0$.
\end{lemma}
\begin{proof}
From the integral equation \eqref{m-integral-1}, we can write
$$(m^{(\pm})_1=e_1 + \lambda I^{(\pm)}(0;\lambda),$$
where
$$I^{(\pm)}(0;\lambda)=\int_{\pm\infty}^0 e^{i\lambda(x-y)\sigma_3}\mathcal{M}(m^{(\pm)})_1e^{-i\lambda(x-y)\sigma_3}dy.$$
We note that since, for every $x\in \RR^{\pm}$, $m^{(\pm)}(x;\lambda)$ is continuous in $\lambda$
, then $I^{(\pm)}(0;\lambda)$ is also continuous. 
From \eqref{b-det}, using the above notation, we have 
$$b(\lambda)=\lambda\{\det(e_1,I^{(-)}(0;\lambda))+\det(I^{(+)}(0;\lambda),e_1)\}+\lambda^2\det(I^{(-)}(0;\lambda),I^{(+)}(0;\lambda)).$$
It is evident that $\lambda^{-1}b(\lambda)$ is continuous at $\lambda=0$. A similar way applies to showing continuity of $\lambda^{-1}(a(\lambda)-1)$ at $\lambda=0$.
\end{proof}

Finally, we obtain the following.
\begin{theorem} \label{b-a-regularity}
If $q\in X_{n+1}$ and the norm $\|q\|_{X_1}$ is sufficiently small, then $a(\lambda)$ and $b(\lambda)$ defined in \eqref{def-T} satisfy 
$$\lambda^{-1}\frac{b}{a}\in H^1, \quad \lambda^{j}\frac{b}{a}\in H^1$$ 
for $j=0,1,\cdots, n$.
\end{theorem}
\begin{proof}
Since $\lambda^jb\in L^2$ for $j=0,\cdots, n$ from Lemma \ref{b-estimate-2} and $a$ is continuous and bounded away from zero from Lemmas \ref{a-H1} and \ref{a-nozero}, it is obvious that $\lambda^j\frac{b}{a} \in L^2$ for $j=0, \cdots,n$. 
Next, we consider, 
$$\partial_{\lambda}\left(\lambda^j\frac{b}{a}\right)=\frac{\partial_{\lambda}(a(\lambda)e^{-i\lambda \int_{\RR}\mathcal{H}dx})}{a(\lambda)^2}\lambda^jb(\lambda)e^{i\lambda \int_{\RR}\mathcal{H}dx}+$$
$$\quad+\frac{e^{-i\lambda \int_{\RR}\mathcal{H}dx}}{a(\lambda)}\partial_{\lambda}(\lambda^jb(\lambda)e^{-i\lambda \int_{\RR}\mathcal{H}dx}).$$
All terms in the right-hand side above are estimated in $L^2$ from Lemmas \ref{b-estimate-2}, \ref{a-H1}, and \ref{a-nozero} for $j=0,1,\cdots, n$ if $q\in X_{n+1}$.

For $\frac{1}{\lambda}\frac{b}{a}$, we only care about $\lambda=0$. From Lemma \ref{continuity-at-zero}, $\frac{1}{\lambda}\frac{b}{a}$ is continuous at $\lambda=0$, so it implies that the weak derivative exists and thus $\frac{1}{\lambda}\frac{b}{a} \in H^1$.
\end{proof}
The following Corollary will be used in the inverse problem. Its proof is apparent from Lemmas \ref{b-estimate-2}, \ref{a-H1}, \ref{a-nozero}, and Proposition \ref{small-solution}, that norms of $b(\lambda)/a(\lambda)$ are controlled in terms of $X_n$ norms of $q$. In particular, we pay our attention to smallness assumption on $\|q\|_{X_1}$. 
\begin{corollary} \label{smallness-a-b}
Let $\delta^*>0$ be a small constant such that $\|q\|_{X_1}<\delta^*$ is small enough in the sense of Theorem \ref{b-a-regularity}.  For any small $\delta_1 >0$, there exists $\delta_2>0$ such that if $\|q\|_{X_1}<\delta_2<\delta^*$, then 
$$\left\|\lambda^{-1}\frac{b(\lambda)}{a(\lambda)}\right\|_{L^2}+\left\|\frac{ b(\lambda)}{a(\lambda)}\right\|_{L^2}<\delta_1,$$
where $a$ and $b$ are defined in \eqref{def-T}.
\end{corollary} 
 
\section{The inverse power transformation $z=-\frac{1}{\lambda}$} \label{inverse-transformation}
Finally, here we put all previous results together to relate solution of the Wadati-Konno-Ichikawa spectral probem to solution of the normalized Riemann-Hilbert problem.
Here we make the change of variable in the spectral variable, 
$$z:=-\frac{1}{\lambda}.$$
It is clear that $\mbox{Im}(z)\gtrless 0$ if and only if $\mbox{Im}({\lambda})\gtrless 0$. 
What follows is that we have the consistent analyticity for $m^{\pm}(z^{-1})$ and $a(z^{-1})$, i.e., domains of analyticity of $m^{\pm}(\lambda)$ in the $\lambda$-complex plane coincide with domains of $m^{\pm}(-z^{-1})$ in the $z$-complex plane. What is better in the $z$-variable is that the sectionally analytic function $m(-1/z)$ is normalized as $z\rightarrow \infty$ and gives the useful reconstruction formula for the potential $q$.
From the integral equation \eqref{m-integral-1}, we obtain the following result that is given in the $z$-variable.
\begin{proposition}  \label{m-analyticity-reconstruct}
If $q\in X_1$, then $[(m^{(+)})_1, (m^{(-)})_2]$ is analytic in $\CC^+_z$, and $[(m^{(-)})_1, (m^{(+)})_2]$ is analytic in $\CC^-_z$. Their asymptotic behaviors at infinity are
$$[(m^{(+)})_1, (m^{(-)})_2] \rightarrow I, \quad [(m^{(-)})_1, (m^{(+)})_2] \rightarrow I$$
as $|z|\rightarrow \infty$ in their domains of analyticity.
Furthermore, the first term $m_1$ in the series 
$$[(m^{(+)})_1, (m^{(-)})_2]=I+\frac{m^{(1)}}{z} + \mathcal{O}(\frac{1}{z^2})$$
gives $$\partial_xm^{(1)}=\mathcal{M}=\begin{pmatrix} 0 & q\\ -\bar{q} & 0 \end{pmatrix}.$$
The same holds for $[(m^{(-)})_1, (m^{(+)})_2]$.
\end{proposition}  

\begin{proposition} \label{a-analytic}If $q\in X_1$, then
$a(-1/z)$ is analytic in $\CC^+$ with 
$$a(-1/z) =1+\mathcal{O}(\frac{1}{z^2}),$$ 
for $z\in \CC^+$ in the neighborhood of $\infty$.
\end{proposition}

We define the reflection coefficient 
\begin{equation}\label{r-def}
r(z):=\frac{b(-1/z)}{a(-1/z)}.
\end{equation}
In above, if we write $r(-1/\lambda)$, then it simply means $r(-1/\lambda)=\frac{b(\lambda)}{a(\lambda)}.$ From Theorem \ref{b-a-regularity}, $\lambda^j r(1/\lambda)$ belongs to $H^1_{\lambda}$ for $-1\leq j \leq n$ if $q\in X_{n+1}$ and $q$ is sufficiently small in $X_1$ norm.

To this end, we define the sectionally analytic matrix function by

\begin{equation} \label{normalized-m}
m(z)=\left\{ \begin{matrix} \left[ \frac{1}{a(-1/z)e^{-\frac{i}{z}\int_{\RR}\mathcal{H}dy}} (m^{(+)})_1e^{-\frac{i}{z}\int_{x}^{\infty}\mathcal{H}dy}, (m^{(-)})_2e^{-\frac{i}{z}\int_{-\infty}^x\mathcal{H}dy}\right] & z \in \CC^+ \\
 [(m^{(-)})_1e^{\frac{i}{z}\int_{-\infty}^x\mathcal{H}dy},  \frac{1}{\overline{a}(-1/\bar{z})e^{\frac{i}{z}\int_{\RR}\mathcal{H}dy}} (m^{(+)})_2e^{\frac{i}{z}\int_{x}^{\infty}\mathcal{H}dy}] & z \in \CC^- 
 \end{matrix} \right.
 \end{equation}
 Recall $\HH=\Jq-1$.
This sectionally analytic matrix is normalized correctly, i.e., $m(z) \rightarrow I$ as $|z|\rightarrow \infty$ and the non-tangential limits of $m(z)$ as $|z|\rightarrow 0$ are bounded. The direct computation can verify that 
the jump condition on the real line is found as 
\begin{equation} \label{jump-RHP}
m_+= m_- \begin{pmatrix} 1+|r|^2 & \bar{r}e^{-\frac{2i}{z}(x+\int_{-\infty}^x\mathcal{H}dy)} \\
r e^{\frac{2i}{z}(x+\int_{-\infty}^x\mathcal{H}dy)} & 1 \end{pmatrix} \quad z\in \RR,
\end{equation}
where $m_{\pm}$ are the non-tangential limits of $m(z)$ to $z\in \RR$ from $\CC^{\pm}$.

At $z=0$, since $r(0)=0$, we have that $m_+=m_-$, i.e., the non-tangential limits coincide at the origin. From Propositions \ref{m-limits} and \ref{a-limits}, the non-tangential limits of $m(z)$ in \eqref{normalized-m} at $z=0$ are 
$$m(z) \rightarrow \mathcal{G}e^{\sigma_3\int_{-\infty}^xBdy}$$
as $z\rightarrow 0$ for $z\in \CC^{\pm}$. This is consistent with $m_+=m_-$ at $z=0$, as we found $m_{\pm}(0)= \mathcal{G}e^{\sigma_3\int_{-\infty}^xBdy}$ above.
 
 However, the jump condition \eqref{jump-RHP} contains the potential $q$ in $\HH$ that we want to recover in the inverse problem. In the latter section, we make a final modification to address appearance of $q$.
 
\section{The RHP after the change of space coordinate} \label{change-space-variable} 
\subsection{Change of space coordinate--new notations $x_{\HH}$ and $q_{\HH}$}  

In order to address appearance of $q$ in the jump condition, we introduce 
\begin{equation} \label{xH}
x_{\HH}:=x+\int_{-\infty}^x\HH(y) dy,
\end{equation}
where $\HH=(\langle q\rangle -1)$.
If $q \in L^1$, $x_{\HH}$ is continuous and monotone increasing in $x$ and $\int_{\RR}\HH dy$ is bounded, so we see that for every $x\in \RR$, there exists a unique corresponding value $x_{\HH} \in \RR$, i.e., $x\mapsto x_{\HH}$ is one-to-one and onto. 
In the following, $x_c$ is a negative value such that $x_{\HH}=x_c+\int_{-\infty}^{x_c}\HH dy=0$. 

\begin{figure}[htbp]
\centering
\begin{tikzpicture}[decoration={markings,
mark=at position 1 with {\arrow[line width=1pt]{>}}}
]

\path[draw, postaction=decorate] (-4,0) node[above] {$-\infty$} -- (4,0) node[above]{$x$};
\path[draw, postaction=decorate] (-4,-2) node[below] {$-\infty$} -- (4,-2) node[below]{$x_{\HH}$};
\path[draw] (-2,0) node[above]{$x_c$} -- (0, -2) node[below]{$0$} [dashed];
\path[draw] (0,0) node[above]{$0$} -- (2, -2) node[below]{$\int_{-\infty}^{0}\HH dy$} [dashed];
\end{tikzpicture}
\end{figure}

 We denote the potential defined on $x_{\HH}$ variable as $q_{\HH}(x_{\HH})$. The relation to $q(x)$ is simply
$$q_{\HH}(x_{\HH})=q_{\HH}(x+\int_{-\infty}^x\HH dy)=q(x),$$
i.e., $q(x)$ is obtained by $q_{\HH}(x)$ after translation $\int_{-\infty}^x\HH dy$ in the argument of $q_{\HH}(x)$.

To this end, we reformulate the jump condition \eqref{jump-RHP} as
\begin{equation} \label{jump-RHP-2}
m_+= m_- \begin{pmatrix} 1+|r|^2 & \bar{r}e^{-\frac{2i}{z}x_{\HH}} \\
r e^{\frac{2i}{z}x_{\HH}} & 1 \end{pmatrix} \quad z\in \RR.
\end{equation}
In the $x_{\HH}$ variable, the reconstruction formula $\partial_xm^{(1)} = \mathcal{M}$ after the change of variable gives
\begin{equation}\label{reconstruct-potential}
q_{\HH}\langle q_{\HH}\rangle^{-1} =\partial_{x_{\HH}}m^{(1)}_{12},
\end{equation}
where we recall $\lim_{|z|\rightarrow\infty, z\in \CC^+}z \;m(z)=m^{(1)}=\begin{pmatrix}m^{(1)}_{11} & m^{(1)}_{12}\\m^{(1)}_{21} & m^{(1)}_{22}\end{pmatrix}$ and $\partial_{x_{\HH}}$ is a partial derivative with respect to $x_{\HH}$.
\subsection{How to construct $x$ in the inverse problem}
However, when the inverse problem concerns, we have no a-priori knowledge of the potential $q$, which means that the variable $x$ must be recovered as well after $q_{\HH}$ is recovered. We proceed as follows.

We notice that if $|\partial_{x_{\HH}}m_{12}^{(1)}|<1$ we can recover $q_{\HH}$ in \eqref{reconstruct-potential}. Indeed, solve for $|q_{\HH}|^2$ in \eqref{reconstruct-potential} and obtain
$$|q_{\HH}|^2=\frac{|\partial_{x_{\HH}}m_{12}^{(1)}|^2}{1-|\partial_{x_{\HH}}m_{12}^{(1)}|^2} <\infty$$
which needs $|\partial_{x_{\HH}}m_{12}^{(1)}|<1$. At this moment, we suppose that $|\partial_{x_{\HH}}m_{12}^{(1)}|<1$ as this will be proven later, then we determine $q_{\HH}$ by
$$
q_{\HH} =\langle q_{\HH}\rangle\partial_{x_{\HH}}m_{12}^{(1)}.
$$ 
In order to recover $x$, from \eqref{xH} and using definition $q(x)=q_{\HH}(x+\int_{-\infty}^x\HH dy)$, we have

$$x_{\HH}=x+\epsilon(x)$$
where
\begin{equation} \label{ode-integral}
\epsilon(x)=\int_{-\infty}^x(\langle q_{\HH}(y+\epsilon(y))\rangle-1)dy.
\end{equation}
To recover $x$ uniquely, a unique existence of $\epsilon$ must be established. 
\begin{lemma} \label{epsilon-lemma}
Suppose that $q_{\HH} \in X_1$. Then there exists a unique monotone increasing bounded continuous solution $\epsilon(x)$ in \eqref{ode-integral} for $x\in (-\infty,\infty)$ with $\epsilon(-\infty)=0$. 
\end{lemma}
\begin{proof}
Construct a sequence $\{\epsilon_k\}$ from
$$\epsilon_k(x)=\int_{-\infty}^x(\langle q_{\HH}(y+\epsilon_{k-1}(y))\rangle-1)dy,$$
and we find that 
\begin{align*}
\epsilon_{k+1}(x)-\epsilon_{k}(x)&=\int_{-\infty}^x(\langle q_{\HH}(y+\epsilon_{k}(y))\rangle-\langle q_{\HH}(y+\epsilon_{k-1}(y))\rangle)dy \\
&\leq \int_{-\infty}^x\left| \frac{\langle q_{\HH}(y+\epsilon_{k}(y))\rangle-\langle q_{\HH}(y+\epsilon_{k-1}(y))\rangle}{\epsilon_{k}(y)-\epsilon_{k-1}(y)} \right| dy \|\epsilon_{k}-\epsilon_{k-1}\|_{L^{\infty}} \\
&\leq \sup_{\delta \in \RR}\left\| \frac{\langle q_{\HH}(\cdot+\delta)\rangle-\langle q_{\HH}\rangle}{\delta} \right\|_{L^1(-\infty,x)}\|\epsilon_{k}-\epsilon_{k-1}\|_{L^{\infty}}.
\end{align*}
If $q\in X_1$, then there exists some $x_0 \in \RR$ such that $M=\sup_{\delta \in \RR}\left\| \frac{\langle q_{\HH}(\cdot+\delta)\rangle-\langle q_{\HH}\rangle}{\delta} \right\|_{L^1(-\infty,x_0)}<1$. We conclude that the sequence $\{\epsilon_k(x)\}$ is Cauchy in $L^{\infty}(-\infty,x_0)$ space since $\|\epsilon_m-\epsilon_n\|_{L^{\infty}(-\infty,x_0)}\leq \sum_{j=n}^{m-1}\|\epsilon_{j+1}-\epsilon_{j}\|_{L^{\infty}(-\infty,x_0)}\leq \frac{M^n}{1-M}\|\epsilon_1-\epsilon_0\|_{L^{\infty}(-\infty,x_0)} \rightarrow 0$ as $n\rightarrow \infty$ for $m>n$. 

Next, in order to extend the existence interval, consider 
$
\epsilon(x)=\epsilon(x_0)+\int_{x_0}^x(\langle q_{\HH}(y+\epsilon(y))\rangle-1)dy
$
and repeat the same argument to conclude a unique existence of $\epsilon(x)$ in $L^{\infty}(x_0,x_1)$ for $x_1>x_0$. Since $q \in X_1$, there are a finite number of subintervals of $\RR$ such that $M<1$ for each interval. By uniqueness, this proves a unique existence of $\epsilon(x)$ for $x \in (-\infty,\infty)$. 

Lastly, the derivative of $\epsilon(x)$ from \eqref{ode-integral} is $\epsilon'(x)=\langle q_{\HH}(x+\epsilon(x))\rangle -1>0$. Since $q_{\HH}\in H^1\subset X_1$, $\epsilon'(x)$ is bounded. This implies that $\epsilon(x)$ is continuous and monotone increasing. 
\end{proof}

With $\epsilon(x)$ given in Lemma \ref{epsilon-lemma}, we see that the variable $x\in \RR$ can be uniquely determined from
$x_{\HH}=x+\epsilon(x)$
with respect to $x_{\HH}\in \RR$. For example, the solid line below corresponds to $x+\epsilon(x)$, 
\begin{figure}[htbp]
\centering
\begin{tikzpicture}[decoration={markings,
mark=at position 1 with {\arrow[line width=1pt]{>}}}
]

\path[draw, postaction=decorate] (-3,0) node[above] {} -- (3,0) node[above]{$x$};
\path[draw, postaction=decorate] (0,-3) node[below] {} -- (0,3) node[right]{$x_{\HH}$};
\path[draw] (-3,-3) -- (3, 3) [dashed];
\draw plot [smooth, tension=0.3] coordinates { (-3,-2.8) (0,0.8) (2, 2.8) (3,3.8) };
\end{tikzpicture}
\end{figure}

\subsection{Useful explicit formula for $x$}
In the previous subsection, we gave a basic procedure to recover $q(x)$ as $q_{\HH}\rightarrow x \rightarrow q(x)$ provided that $q_{\HH} \in X_1$. Here, we will give the formula that relates $x$ and $x_{\HH}$ in terms of the limit of $m(z)$ defined in \eqref{normalized-m}. We will not use this to estimate $q(x)$, as the procedure above is sufficient, but the following formula is particularly useful in study of long time asymptotic solutions and explicit soliton solutions.

We denote $(m(z))_{11}$ as (1,1) th matrix element of $m(z)$ defined in \eqref{normalized-m}. From Propositions \ref{m-analyticity-reconstruct} and \ref{a-analytic}, we easily see that 
$$(m(z))_{11}=1+\frac{i}{z}\int_{-\infty}^{x}(\langle q(y)\rangle -1)dy+\mathcal{O}(z^{-2}).$$
From this, we obtain
\begin{equation} \label{x-explicit}
\lim_{z\rightarrow \infty} z((m(z))_{11}-1)=i\int_{-\infty}^x\HH dy,
\end{equation}
where $z\rightarrow \infty$ is taken in the upper-half plane of $\CC$. From \eqref{xH} and \eqref{x-explicit}, we have the relation
$$x_{\HH}-x=\frac{1}{i}\lim_{z\rightarrow \infty} z((m(z))_{11}-1).$$
The right-hand side can be determined from solution of the Riemann-Hilbert problem. This is why it is useful for estimates as well as explicit soliton solutions. We will use \eqref{x-explicit} to derive soliton solution. 
 
\subsection{Final RHP formulation with time parameter}  
Up to now, we consider the case of $t=0$. Here we shall give the time evolution of $b/a$ under the time evolution of $q(x,t)$ to set up the RHP including the time parameter $t$. 

The existence of local unique solution $q\in X_{\infty}$ for $t\in [0,T)$ for some $T>0$ can be shown by the most rudimentary way. We write the WKI equation \eqref{WKI} as
$$q=W_t(q), \quad W_t(q):=-iq_0-i\int_0^t\left(\frac{q}{\Jq}\right)_{xx}d\tau.$$
We see that $W_t: X_{m+2}\rightarrow X_m$ with $m \geq 0$ and the fixed-point argument can be closed in the space $X_{\infty}$ for a finite time interval $[0,T)$ with some $T>0$. 

For every fixed $t\in [0,T)$, considering the Lax system \eqref{WKI-spectral} with the potential $q(\cdot,t) \in X_{\infty}$ of the WKI equation, we deduce that the exactly same results in Section \ref{WKI-AKNS} hold for the fundamental solutions $\psi_{\pm}(x,t;\lambda)$. On the other hand, considering the other Lax system \eqref{WKI-spectral-t} with $q(x,t)$ and taking $|x|\rightarrow \infty$, we obtain $\psi_t=-2i\lambda^2\sigma_3\psi$. This implies that the fundamental solutions for the both Lax systems must take form of $\psi_{\pm}(x,t;\lambda)e^{-2i\lambda^2t\sigma_3}$.

Following the exactly same definition of the scattering coefficients $a$ and $b$ with $\psi_{\pm}(x,t,;\lambda)$ in \eqref{T-matrix-a-b}, we find 
$$a(\lambda,t) =\det((\psi_+)_1,(\psi_-)_2)=\det(e^{-2i\lambda^2t}(\psi_+)_1,e^{2i\lambda^2t}(\psi_-)_2)=a(\lambda),$$
$$b(\lambda,t) =\det((\psi_-)_1,(\psi_+)_1)=e^{4i\lambda^2t}\det(e^{-2i\lambda^2t}(\psi_-)_1,e^{-2i\lambda^2t}(\psi_+)_1)=e^{4i\lambda^2t}b(\lambda).$$
The last equalities are due to the fact that the traces of the Lax operators are zeros and that $\psi_{\pm}(x,t;\lambda)e^{-2i\lambda^2t\sigma_3}$ are the fundamental solutions. We obtain 
$$\frac{b(\lambda,t)}{a(\lambda,t)}=\frac{b(\lambda)}{a(\lambda)}e^{4i\lambda^2t},$$
and for the reflection coefficient $r$ defined in \eqref{r-def}, we write
$$r(z,t)=r(z)e^{4i\frac{t}{z^2}},$$
which corresponds to the time evolution $q_0\mapsto q(x,t)$ of the WKI equation \eqref{WKI}. 
The jump condition \eqref{jump-RHP-2} for $t\in [0,T)$ is finally formulated as 
\begin{equation} \label{jump-RHP-3}
m_+= m_- \begin{pmatrix} 1+|r|^2 & \bar{r}e^{-\frac{2i}{z}x_{\HH}-\frac{4i}{z^2}t} \\
r e^{\frac{2i}{z}x_{\HH}+\frac{4i}{z^2}t} & 1 \end{pmatrix} \quad z\in \RR.
\end{equation}
In the inverse problem, we will use the above formulation. 
The scheme of the inverse problem works as follows. 
\begin{itemize}
\item show $|\partial_{x_{\HH}}m_{12}^{(1)}|<1$ (Lemma \ref{smallness-m})
\item recover $q(x_{\HH},t)$ through the reconstruction formula \eqref{reconstruct-potential}
\item estimate $q_{\HH}$ in $X_{m}$ space for some $m\geq 1$ for every $t\geq 0$ (Lemma \ref{qH-construct})
\item recover the variable $x$ (Lemma \ref{x-construct})
\item estimate $q(\cdot,t)\in X_m$ in the $x$ variable for every $t\geq 0$. (Theorem \ref{construct-q})
\end{itemize}
We will conclude that the the maximal existence time $T$ for a local solution $q(\cdot,t)\in X_{\infty}$ for $t\in [0,T)$ can be extended to an arbitrary large number. 
\section{The inverse problem} \label{RHP-section}
\subsection{Preliminaries and notations}
The sectionally analytic matrix function $m(z)$ in \eqref{normalized-m} is viewed as solution of the normalized RHP. In the inverse problem, we are going backwards, i.e., we start with a given data $r(z)$, obtain $m(z)$, and construct $q(z)$. We first define the normalized Riemann-Hilbert problem and introduce notations and concepts used in the following subsections. 
 
 Let a general contour $\Sigma \subset \CC$ be a finite union of simple smooth curves that can be either closed on the $\CC$ plane or extended to be closed on the Riemann sphere. The complex plane $\CC$ is divided into finitely many components and a union of their boundaries coincide with $\Sigma$. We say that a component is positively (negatively) oriented when its boundary is oriented positively (negatively). A positive orientation is a clockwise direction. We say that for a given contour $\Sigma$ and the $2\times 2$ matrix $v, v^{-1} \in L^{\infty}$, the $2\times 2$ matrix $m(z)$ solves the normalized RHP $(\Sigma, v)$ if 
\begin{itemize}
\item $m(z)$ is analytic in $\CC\setminus \Sigma$
\item $m_+=m_-v \quad z\in \Sigma$
\item $m\rightarrow I$ as $|z| \rightarrow \infty$ 
\end{itemize}
where $m_{\pm}(z)$ are non-tangential limits of $m(z)$ as $z$ approaches to $\Sigma$ from $\pm$ oriented components.


 In order to approach the normalized RHP above, we define projection operators $C^{\pm}_{\Sigma}$ that are the non-tangential limits of the Cauchy operator $C_{\Sigma}$,
 $$C_{\Sigma}\phi=\frac{1}{2\pi i}\int_{\Sigma}\frac{\phi}{s-z}ds$$
as $z$ approaches $\Sigma$ from $\pm$ oriented components in $\CC$. To summerize the notations, we use the subscript $\Sigma$ for the operator $C_{\Sigma}$ as the integral integrates over $\Sigma$, and the superscript $\pm$ to denote the projection operators $C_{\Sigma}^{\pm}$ accordingly. 

As operators $C_{\Sigma}^{\pm}$ are bounded on $L^p$ $(1<p<\infty)$,  we say, more precisely, that $m_{\pm}$ solves the normalized RHP above in $L^p$ sense if $m_{\pm}-I \in L^p$ and $m_+=m_-v$ for $z \in \Sigma$. Indeed, if so, $m_{\pm}$ can be written as $m_{\pm}-I=C^{\pm}_{\Sigma}m_-v$, which satisfy $m_+=m_-v$ by the property that $C^+_{\Sigma}-C^-_{\Sigma}=I$. Its analytic extension is given as $m(z)=I+C_{\Sigma}m_-v$ that satisfies the conditions of the normalized RHP above.

We can solve the normalized RHP in $L^p$ sense by addressing the integral equation $m_{\pm}-I=C^{\pm}_{\Sigma}m_-v$. This integral equation can take alternative forms by exploiting factorizations of $v$, which are useful for various estimates. When a matrix $v$ admits a factorization $v=v_-^{-1}v_+$ with $v_{\pm}, v_{\pm}^{-1} \in L^{\infty}$, we denote $w_{+}=-I+v_+$ and $w_-=I-v_-$, and a pair $w=(w_+,w_-)$ .  
We use the $\theta$ subscript to denote $v_{\theta\pm}:= e^{-i\theta  \sigma_3}v_{\pm}e^{i\theta \sigma_3}, \quad w_{\theta\pm}:= e^{-i\theta  \sigma_3}w_{\pm}e^{i\theta \sigma_3}$ where $\theta$ is a real function on $z$ and will be chosen later for the WKI case.
 Introduce the operator $C_{w_{\theta},\Sigma}$ given by
 \begin{equation} \label{singular-integral-operator}
C_{w_{\theta},\Sigma} \varphi:=C^+_{\Sigma}(\varphi w_{\theta-})+C^-_{\Sigma}(\varphi w_{\theta+}).
\end{equation}
We will address the integral equation
\begin{equation} \label{mu-integral-sigma} 
\mu = I + C_{w_{\theta},\Sigma}\mu.
\end{equation}
The subscript $w_{\theta}$ for the operator $C_{w_{\theta},\Sigma}$ corresponds to $w_{\theta}=(w_{\theta+},w_{\theta-})$ which appears in definition of $C_{w_{\theta},\Sigma}$. 

If the operator $I-C_{w_{\theta},\Sigma}$ is a bijection on $L^p$, $1<p<\infty$, then we have a unique solution $\mu-I \in L^p$ of \eqref{mu-integral-sigma}. Defining $m_{\pm}:= \mu v_{\theta \pm}$, we can easily verify that $m_+=m_-e^{-i \theta \sigma_3}ve^{i \theta \sigma_3}$ for $z\in \Sigma$. In the following subsection, we will consider the specific examples of $\Sigma$ and $v$ to address the normalized RHP for the WKI spectral problem. We will discuss solvability and obtain various estimates from the integral form \eqref{mu-integral-sigma}.
\subsection{The normalized RHP for the WKI problem}
 We are interested in the normalized RHP $(\RR, e^{-i \theta  \sigma_3}ve^{i\theta  \sigma_3})$, 
where
$$v=\begin{pmatrix} 1+|r|^2 & \bar{r}e^{-2 i \theta} \\
r e^{2i \theta} & 1 \end{pmatrix},\quad \theta =  \frac{x_{\HH}}{z}+2\frac{t}{z^2} \quad z \in \RR,$$
and $r(z)$ satisfies 
$$\lambda^j r(1/\lambda) \in H^1_{\lambda}$$
for $-1\leq j \leq m$ for some $m \in \mathbb{N}$. The above case coincide with \eqref{jump-RHP-3} and Theorem \ref{b-a-regularity}. To avoid confusion, we must emphasize that the regularity of the reflection coefficient $r$ is given in terms of $\lambda$ but not $z=-\frac{1}{\lambda}$. To remind us, we will keep using notation $H_{\lambda}^1$ to denote for the $H^1$ space with respect to $\lambda$ variable. 
 In previous Sections, we have shown that for a given potential $q$, we have the sectionally analytic matrix function $m(z)$ given in \eqref{normalized-m} that is solution to the normalized RHP $(\RR,e^{-i \theta  \sigma_3}ve^{i\theta  \sigma_3})$ above. We have shown that the reflection coefficient $r$ is estimated in norms in terms of norms of $q$. For this section, on the other hand, we want to estimate $q$ in terms of $r$ to obtain a global estimate on $q$ in time.  

The matrix $v$ can be factorized as $v=v_{-}^{-1}v_+$ with
\begin{equation} \label{v-triangle}
v_-=\begin{pmatrix} 1 & -\bar{r} \\ 0 & 1 \end{pmatrix}, \quad  v_+=\begin{pmatrix} 1 & 0 \\ r & 1 \end{pmatrix}.
\end{equation}
The above triangulations are used in estimates for $x\in \RR_-$. For the other half-line $\RR_+$, we need to consider different triangulations. To shorten our presentation, we give only the sketch at the end after Lemma \ref{qH-construct}. 
We address the integral equation
\begin{equation}\label{mu-integral-RHP}
\mu = I + C_{w_{\theta},\RR}\mu,
\end{equation}
where $ C_{w_{\theta},\RR}$ is defined in \eqref{singular-integral-operator} with $v_{\pm}$ in \eqref{v-triangle}.

Since $v=v_-^{-1}v_+$ is hermitian with positive eigevalues and $v,v^{-1} \in L^{\infty}$, the solvability of \eqref{mu-integral-RHP} is well-known. We summarize the argument in \cite{Zhou-1989}. To show that $I-C_{w_{\theta},\RR}$ is Fredholm, one notices that there is another bounded operator $I - C_{\tilde{w}_{\theta},\RR}$ on $L^2$ with $\tilde{w}=(-w_+, -w_-)$ such that $(I-C_{\tilde{w}_{\theta},\Sigma})(1-C_{w_{\theta},\Sigma})=(I-C_{w_{\theta},\Sigma})(1-C_{\tilde{w}_{\theta},\Sigma})=I-T$ where $T$ is a compact operator on $L^2$. This implies that $I - C_{w_{\theta},\RR}$ is Fredholm on $L^2$. To show that the Fredholm index is zero, since the Fredholm index is invariant under continuous deformations of Fredholm operators, let $\epsilon w=(\epsilon w_+, \epsilon w_-)$ for $\epsilon \in [0,1]$ and find that $\mbox{Ind}(I-C_{\epsilon w_{\theta},\Sigma})=\mbox{Ind}(I)=0$. Therefore, $I - C_{w_{\theta},\RR}$ is a Fredholm operator of the index zero. Lastly, the zero dimensional kernel of $I - C_{w_{\theta},\RR}$ is shown by the vanishing lemma (Theorem 9.3 in \cite{Zhou-1989}). The Fredholm alternative can be applied to deduce a unique solution $\mu-I \in L^2$ in \eqref{mu-integral-RHP} since $C_{w_{\theta},\RR}I \in L^2$.

The estimate on the inverse of $I - C_{w_{\theta},\RR}$ relies again on the fact that the matrix $v$ is hermitian and has strictly positive eigenvalues. The estimate works for all possible factorizations $v=\tilde{v}_-^{-1}\tilde{v}_+$ with $\tilde{v}_{\pm}, \tilde{v}_{\pm}^{-1} \in L^{\infty}$.
\begin{proposition} \label{inverse}
Let $\|w_{\pm}\|_{L^{\infty}} \leq \eta$ for some positive finite constant $\eta$. Then, 
$$\|(I-C_{w_{\theta},\RR})^{-1}\|_{L^2(\RR)} \leq k$$
where a constant $k>0$ depends only on $\eta$. 
\end{proposition}
The above bound is independent of $x_{\HH}$ and $t$. As discussed before, using solution $\mu$ in \eqref{mu-integral-RHP}, we write $m_{\pm}:=\mu v_{x\pm}$. From \eqref{mu-integral-RHP} and the fact that $C_{\RR}^+-C_{\RR}^-=I$, we realize that $m_{\pm}$ satisfy
$$m_+=m_- v_{\theta} \quad z \in \RR.$$
An analytic extension of $m_{\pm}$ is given as 
\begin{equation}\label{m-analytic-RHP}
m(z) = I + C_{\RR}\mu(w_{\theta+}+w_{\theta-}) \quad z\in \CC\setminus \RR.
\end{equation}

The above discussion implies that, defining $\mu$ as $\mu=m_{\pm}v_{\theta\pm}^{-1}$ where $m_{\pm}$ correspond with the non-tangential limits of the sectionally analytic matrix function $m(z;x_{\HH},t)$ in \eqref{normalized-m}, we obtain the expression of $m$ in terms of $\mu$ as given in \eqref{m-analytic-RHP}. Following the same notation in Proposition \ref{m-analyticity-reconstruct}, we recall
$$m^{(1)}= \lim_{z\rightarrow \infty}z \begin{pmatrix} 0 & (m)_{12} \\ (m)_{21} & 0 \end{pmatrix},$$
so we obtain, from \eqref{m-analytic-RHP}, 
\begin{equation} \label{Reconstruct-formula-inverse} 
\partial_{x_{\HH}}m^{(1)}=\frac{-\sigma_3}{2\pi i} \partial_{x_{\HH}}\left[\frac{\sigma_3}{2} ,\; \int_{\RR}\mu(w_{\theta-}+w_{\theta+})ds\right],
\end{equation}
where $[\cdot,\cdot]$ is a Lie bracket, and $\frac{1}{2}\sigma_3[\sigma_3,A]$ is just the off-diagonal part of $A$.
The (1,2) th element of the above formula is an alternative form of \eqref{reconstruct-potential}.  From this, we will estimate $q_{\HH}$ in terms of the solution $\mu$. 
We will obtain estimates of $\mu$ from the integral equation \eqref{mu-integral-RHP}, i.e., 
\begin{equation} \label{mu-integral-2}
\mu-I=(I-C_{w_{\theta},\RR})^{-1}C_{w_{\theta},\RR}I
\end{equation}
and by differentiating \eqref{mu-integral-RHP} in $x$ after $n$ times, we also have 
\begin{equation} \label{mu-x-integral}
\partial_{x_{\HH}}^n\mu=(I-C_{w_{\theta},\RR})^{-1}\left(n\sum_{j=1}^{n-1}C_{\partial_{x_{\HH}}^jw_{\theta},\RR}\partial_{x_{\HH}}^{n-j}\mu+ C_{\partial_{x_{\HH}}^nw_{\theta},\RR}\mu\right).
\end{equation}
From \eqref{mu-x-integral}, we will use the boostrap argument once we obtain the estimate of $\mu$.

We shall first estimate $C_{w_{\theta},\RR}I$ in \eqref{mu-integral-2}. The following estimate is for the case of $t=0$. Again, we remind that regularity of $r$ is given in $\lambda$ variable. 

In the following, $A\lesssim B$ means that there exists a constant $M>0$ independent of $x_{\HH}$ and $t$ such that $A\leq M B$.
\begin{lemma} \label{estimate-C}
If $\lambda^{j} r(\lambda^{-1}) \in H^1_{\lambda}$ for $-1 \leq j \leq m$ with $m\geq -1$, then
$$\|\partial_{x_{\HH}}^{k} C_{w_{x_{\HH}},\RR}I \|_{L^2} \lesssim (1+x_{\HH}^2)^{-1/2} \quad x_{\HH}\leq  0$$
for all $0\leq k \leq m+1$.
\end{lemma}

\begin{proof}
We shall show the estimate of $C^-_{\RR}w_{x_{\HH}+}$ since that of $C^+_{\RR}w_{x_{\HH}-}$ is done in the same way. We shall consider the following
\begin{align*}
(C^-_{\RR}w_{x_{\HH}+})_{21} &=\frac{1}{2\pi i} \lim_{\epsilon \rightarrow 0}\int_{\RR}\frac{r(s) e^{2i\frac{x_{\HH}}{s}}}{s-(z-i\epsilon)}ds \\
& = \frac{1}{2\pi i} \lim_{\epsilon \rightarrow 0}\int_{\RR}\frac{\frac{1}{\lambda}r(\frac{1}{\lambda}) e^{2i\lambda x_{\HH}}}{1-\lambda(z-i\epsilon)}d\lambda \\
& = \frac{1}{i2(\pi)^{3/2} } \lim_{\epsilon \rightarrow 0} \int_{\RR}\frac{e^{i\lambda(\xi+2x_{\HH})}}{1-\lambda(z-i\epsilon)}d\lambda \mathcal{F}(\lambda^{-1}r(\lambda^{-1}))(\xi)d\xi  \\
& = \frac{1}{\sqrt{\pi}}\frac{1}{z}\int_{-2x_{\HH}}^{\infty} e^{i\frac{1}{z}(\xi+2x_{\HH})}  \mathcal{F}(\lambda^{-1}r(\lambda^{-1}))(\xi)d\xi.  
\end{align*}
The second equality is the change of variable, and the third equality is the Fourier inverse of the Fourier transform of $f$, i.e., $\mathcal{F}(f)=\frac{1}{\sqrt{\pi}}\int_{\RR} f(\lambda)e^{-i\lambda \xi}d\lambda$ and $f=\frac{1}{\sqrt{\pi}}\int_{\RR} \mathcal{F}(f)(\xi)e^{i\lambda\xi}d\xi$, and the fourth equality is by the complex contour integration.  
We find that 
\begin{align} \label{C-estimate-main}
&\| (C^-_{\RR}w_{x_{\HH}+})_{21} \|_{L^2}^2  \\
&= \frac{1}{\pi} \int_{\RR} \left( \int_{-2x_{\HH}}^{\infty} \int_{-2x_{\HH}}^{\infty} e^{-i\frac{1}{z}(\xi-\xi')} \mathcal{F}(\lambda^{-1}r(\lambda^{-1}))(\xi) \overline{\mathcal{F}(\lambda^{-1}r(\lambda^{-1}))(\xi')} d\xi d\xi'\right) \frac{dz}{z^2} \nonumber \\
&=\frac{1}{\pi} \int_{\RR} \left( \int_{-2x_{\HH}}^{\infty} \int_{-2x_{\HH}}^{\infty} e^{-iz(\xi-\xi')} \mathcal{F}(\lambda^{-1}r(\lambda^{-1}))(\xi) \overline{\mathcal{F}(\lambda^{-1}r(\lambda^{-1}))(\xi')} d\xi d\xi' \right)dz \nonumber\\
&=2\int_{-2x_{\HH}}^{\infty} |\mathcal{F}(\lambda^{-1}r(\lambda^{-1}))(\xi)|^2 d\xi. \nonumber
\end{align} 
The second equality is the change of variable and the third equality is the Plancherel's theorem.   
Since $\lambda^{-1}r(\lambda^{-1}) \in H_{\lambda}^{1}$, then for $x_{\HH}\leq 0$, we have 
\begin{equation}\label{key-ineq1}
\int_{-2x_{\HH}}^{\infty} |\mathcal{F}(\lambda^{-1}r(\lambda^{-1}))(\xi)|^2 d\xi \leq \int_{-2x_{\HH}}^{\infty}  \frac{\langle \xi \rangle^{2}}{\langle x_{\HH} \rangle^{2}} |\mathcal{F}(\lambda^{-1}r(\lambda^{-1}))(\xi)|^2 d\xi\lesssim \langle x_{\HH}\rangle^{-2}\|\lambda^{-1}r(\lambda^{-1})\|_{H^1_{\lambda}}^2
\end{equation}
Lastly, consider 
$$\partial_{x_{\HH}}^{m+1}(C^-_{\RR}w_{x+})_{21} =\frac{(2i)^m}{2\pi i} \lim_{\epsilon \rightarrow 0}\int_{\RR}\frac{s^{-m-1}r(s) e^{2i\frac{x_{\HH}}{s}}}{s-(z-i\epsilon)}ds=\frac{(-2i)^m}{2\pi i} \lim_{\epsilon \rightarrow 0}\int_{\RR}\frac{\lambda^{m}r(1/\lambda) e^{2i \lambda x_{\HH}}}{1-\lambda(z-i\epsilon)}d\lambda.$$
The rest follows in the exactly same way as done in \eqref{C-estimate-main} and \eqref{key-ineq1}, and we easily see that $\lambda^{m}r(1/\lambda) \in H^1_{\lambda}$ is needed for $k=m+1$.
\end{proof}

The previous Lemma for the case $\partial_x^{m+1}C_{w_{x_{\HH}},\RR}I$ requires $H^1_{\lambda}$ property of $\lambda^m r(1/\lambda)$. 
Now, when we consider the time evolution, that is, $\lambda^m r(1/\lambda)e^{4i\lambda^2t}$, we notice that $\lambda^m r(1/\lambda)e^{-4i\lambda^2t}$ is $H^1_{\lambda}$ function if $\lambda^m r(1/\lambda) \in H^1_{\lambda}$ and $\lambda^{m+1}r(1/\lambda) \in L^2_{\lambda}$. We deduce the following.
\begin{lemma} \label{estimate-C-time}
If $\lambda^{j} r(\lambda^{-1}) \in H^1_{\lambda}$ for $-1 \leq j \leq m$ and $\lambda^{m+1}r(\lambda^{-1})\in L^2_{\lambda}$ with $m\geq -1$, then
$$\|\partial_{x_{\HH}}^{k} C_{w_{\theta},\RR}I \|_{L^2} \leq K(t)(1+x_{\HH}^2)^{-1/2} \quad x_{\HH}\leq 0$$
for $0\leq k\leq m+1$,
where $K(t)$ grows at most polynomially in $t$. 
\end{lemma}

From the integral equations \eqref{mu-integral-2} and \eqref{mu-x-integral}, and the previous Lemma, we obtain the following.
\begin{lemma} \label{mu-estimate-lem}
If ${\lambda}^{j}r(\lambda^{-1})\in H^1_{\lambda}$ for $-1 \leq j \leq 0$, then 
\begin{equation} \label{mu-estimate}
\|\mu-I\|_{L^2} \lesssim \langle x_{\HH} \rangle^{-1}, \quad x_{\HH}\leq 0,
\end{equation}
and furthermore if ${\lambda}^{j}r(\lambda^{-1})\in H^1_{\lambda}$ for $-1 \leq j \leq m+1$ with $m\geq 0$, then 
\begin{equation}\label{mu-x-estimate}
\|\partial^{k}_{x_{\HH}}\mu\|_{L^2} < K(t) \langle x_{\HH} \rangle^{-1}, \quad x_{\HH}\leq 0
\end{equation}
for $1\leq k\leq m+1$, where $K(t)$ grows at most polynomially in $t$.
\end{lemma}
\begin{proof}
From \eqref{mu-integral-2} and Proposition \ref{inverse} for the inverse $(I-C_{w_{\theta},\RR})^{-1}$ and Lemma \ref{estimate-C} for $C_{w_{\theta},\RR}I$, if $\lambda^jr(1/\lambda)\in H^1_{\lambda}$ for $-1\leq j \leq 0$ we have 
\begin{equation} \label{mu-ineq1}
\|\mu-I\|_{L^2} \lesssim \|C_{w_{\theta},\RR}I\|_{L^2}\lesssim \langle x_{\HH}\rangle^{-1} \quad x_{\HH}\leq 0,
\end{equation}
which proves \eqref{mu-estimate}.

Similarly, from \eqref{mu-x-integral}, we have
\begin{align} \label{mu-ineq2}
\|\partial_{x_{\HH}}\mu\|_{L^2} &\lesssim \|C_{\partial_{x_{\HH}} w_{\theta},\RR}(\mu-I)-C_{\partial_{x_{\HH}} w_{\theta},\RR}I\|_{L^2} \nonumber\\
&\leq  \|C_{\partial_{x_{\HH}} w_{\theta},\RR}(\mu-I)\|_{L^2}+\|C_{\partial_{x_{\HH}} w_{\theta},\RR}I\|_{L^2}.
\end{align}
The second term of the right-hand side in \eqref{mu-ineq2} is estimated from Lemma \ref{estimate-C-time} if $r(1/\lambda)\in H^1_{\lambda}$. The first term above is estimated as
 $$\|C_{\partial_{x_{\HH}} w_{\theta},\RR}(\mu-I)\|_{L^2} \lesssim \|z^{-1}r(z)\|_{L^{\infty}_z}\|\mu-1\|_{L^2}, \quad x_{\HH}\leq 0.$$
If $\lambda r(1/\lambda) \in H^{1}\subset L^{\infty}$, then it implies that $z^{-1} r(z) \in L^{\infty}$. We showed \eqref{mu-x-estimate} in the case of $m=1$, as we need $\lambda^jr(1/\lambda)\in H^1_{\lambda}$ for $-1\leq j \leq 1$. 

Next, to estimate $\partial_{x_{\HH}}^2\mu$ in $L^2_z$ norm, from \eqref{mu-x-integral}, we need results of \eqref{mu-ineq1} and \eqref{mu-ineq2}. We also have additional terms $C_{\partial_{x_{\HH}}^2 w_{\theta},\RR}I$ and $z^{-2}r(z)\in L^{\infty}_z$. Following the exactly same argument above, we need $\lambda^jr(1/\lambda) \in H_{\lambda}^1$ for $-1\leq j \leq 2$ for \eqref{mu-x-estimate} with $m=2$.
The rest follows from the inductive argument.

\end{proof}

\subsection{Estimates on the reconstruction formula \eqref{Reconstruct-formula-inverse} }
From results above, we obtain the following result.
\begin{lemma} \label{estimate-formula}
If ${\lambda}^{j}r(\lambda^{-1})\in H^1_{\lambda}$ for $-1 \leq j \leq m+1$ with $m\geq 0$, then 
$$\left\|\langle x_{\HH}\rangle \partial_{x_{\HH}}^{k} \int_{\RR}\mu(w_{\theta-}+w_{\theta+})ds \right\|_{L_{x_{\HH}}^2(\RR^-)} \leq K(t)$$
for $1 \leq k \leq m+1$,
where $K(t)$ grows at most polynomially in $t$.
\end{lemma}
\begin{proof}
We make the following decomposition,
\begin{align} \label{mu-x-decomp}
\partial_{x_{\HH}}\int_{\RR}\mu(w_{\theta-}+w_{\theta+})ds & =\int_{\RR}\mu(\partial_{x_{\HH}}w_{\theta-}+\partial_{x_{\HH}}w_{\theta+})ds+\int_{\RR}(\partial_{x_{\HH}}\mu)(w_{\theta-}+w_{\theta+})ds.
\end{align}
The first term in the right-hand side of \eqref{mu-x-decomp} is decomposed as
\begin{equation} \label{mu-x-decomp-1} 
\int_{\RR}\mu(\partial_{x_{\HH}}w_{\theta-}+\partial_{x_{\HH}}w_{\theta+})ds  = I + II + III
\end{equation}
with
\begin{align*}
I&=\int_{\RR} (\partial_{x_{\HH}}w_{\theta-}+\partial_{x_{\HH}}w_{\theta+})ds \\
II &= \int_{\RR}[(C_{w_{\theta},\RR}I)](\partial_{x_{\HH}}w_{\theta-}+\partial_{x_{\HH}}w_{\theta+})ds\\
III &= \int_{\RR}[C_{w_{\theta},\RR}(I-C_{w_{\theta},\RR})^{-1}(C_{w_{\theta},\RR}I)](\partial_{x_{\HH}}w_{\theta-}+\partial_{x_{\HH}}w_{\theta+})ds. 
\end{align*}
First, we notice that by the Fourier transform, $I$ can be expressed as
\begin{equation} \label{I-Fourier}
I=2i\sqrt{\pi}\begin{pmatrix} 0 & \mathcal{F}(\frac{1}{\lambda}\overline{r}(1/\lambda)e^{-4i\lambda^2t})(-2x_{\HH}) \\
\mathcal{F}(\frac{1}{\lambda}r(1/\lambda)e^{4i\lambda^2t})(2x_{\HH}) & 0 \end{pmatrix}.
\end{equation}
Therefore, 
$$\|\langle x_{\HH}\rangle I\|_{L^2} \lesssim \| \langle x_{\HH}\rangle \mathcal{F}(\frac{1}{\lambda}r(1/\lambda)e^{4i\lambda^2t})(x_{\HH})\|_{L^2} \lesssim \|\frac{1}{\lambda}r(1/\lambda)e^{4i\lambda^2t}\|_{H^1_{\lambda}}$$
that is bounded if $\frac{1}{\lambda}r(1/\lambda) \in H^1_{\lambda}$ and $r(1/\lambda) \in L_{\lambda}^2$.

For $II$, we use properties of the strictly triangular matrices $w_{\theta \pm}$, i.e., $(C^{\pm}w_{\theta\mp})w_{\theta\mp}=0$, so
\begin{align} 
II &= \int_{\RR}(C^+w_{\theta-})\partial_{x_{\HH}}w_{\theta+} + \int_{\RR}(C^-w_{\theta+})\partial_xw_{\theta-} \nonumber\\
   &= -\int_{\RR}(C^+w_{\theta-})(C^-\partial_{x_{\HH}}w_{\theta+}) + \int_{\RR}(C^-w_{\theta+})(C^+\partial_{x_{\HH}}w_{\theta-} ), \label{II-trick}
\end{align}
where the last quality used $(C^+)^2=C^+$ and $(C^-)^2=-C^-$. From Lemma \ref{estimate-C-time}, $\| \langle x \rangle II\|_{L^2(\RR_-)}$ can be estimated by the Minkowski inequality and Lemma \ref{estimate-C-time} if $\lambda^jr(\lambda^{-1})\in H^1_{\lambda}$ for $-1\leq j \leq 0$ and $\lambda r(\lambda^{-1})\in L^2_{\lambda}$, that is, 
$$\|\langle x_{\HH}\rangle II\|_{L^2_{x_{\HH}}(\RR^-)}\lesssim \left\| \langle x_{\HH}\rangle \|C^{\pm}w_{\theta\mp}\|_{L^2_z}\right\|_{L^{\infty}_{x_{\HH}}(\RR^-)}\left\| \|C^{\pm}\partial_{x_{\HH}}w_{\theta\mp}\|_{L^2_z}\right\|_{L^{2}_{x_{\HH}}(\RR^-)}.$$

For $III$, we use the same way as done for $II$. Writing 
$C_{w,\RR}(I-C_{w,\RR})^{-1}(C_{w,\RR}I)=C_{w,\RR}(\mu-I)$, 
we use the same trick in \eqref{II-trick}, i.e.,
\begin{align}
III &=\int_{\RR}[C^+(\mu-I)w_{\theta-}]w_{\theta+}+\int_{\RR}[C^-(\mu-I)w_{\theta+}]w_{\theta-} \label{III}\\
&=\int_{\RR}[C^+(\mu-I)w_{\theta-}](C^-w_{\theta+})+\int_{\RR}[C^-(\mu-I)w_{\theta+}](C^+w_{\theta-}). \nonumber
\end{align}
 Thus, $\| \langle x_{\HH} \rangle III\|_{L^2(\RR^-)}$ is estimated by Lemmas \ref{estimate-C-time} and \eqref{mu-estimate} if $\lambda^jr(\lambda^{-1})\in H^1_{\lambda}$ for $-1\leq j \leq 0$, that is,
$$\| \langle x_{\HH} \rangle III\|_{L^2(\RR^-)}$$
$$\lesssim \|w_{\theta\pm}\|_{L^{\infty}} \left\| \langle x_{\HH}\rangle \|C^{\pm}w_{\theta\mp}\|_{L^2_z}\right\|_{L^{\infty}_{x_{\HH}}(\RR^-)}\left\| \|\mu-I\|_{L^2_z}\right\|_{L^{2}_{x_{\HH}}(\RR^-)}.$$

The second term in the right-hand side of \eqref{mu-x-decomp} is decomposed as
\begin{align} 
\int_{\RR}(\partial_x\mu)(w_{x-}+w_{x+})ds &= \int_{\RR}[(I-C_{w,\RR})^{-1}C_{\partial_xw,\RR}\mu](w_{x-}+w_{x+})ds \nonumber\\
&= I'+II'+III' \label{second-term}
\end{align}
with
\begin{align*}
I'&=\int_{\RR}[C_{\partial_xw,\RR}I](w_{x-}+w_{x+})ds \\
II'&=\int_{\RR}[C_{\partial_xw,\RR}(\mu-I)](w_{x-}+w_{x+})ds \\
III'&= \int_{\RR}[C_{w,\RR}(I-C_{w,\RR})^{-1}C_{\partial_xw,\RR}\mu](w_{x-}+w_{x+})ds. 
\end{align*}
We see that $\|\langle x_{\HH}\rangle I'\|_{L^2_{x_{\HH}}}$ is estimated in the same way as done in $II$ which requires $\lambda^jr(\lambda^{-1})\in H^1_{\lambda}$ for $-1\leq j \leq 0$ and $\lambda r(\lambda^{-1})\in L^2_{\lambda}$. 

For $II'$, we used the same way as done in \eqref{III}, and we see that, to estimate $\|\langle x_{\HH}\rangle II'\|_{L^2_{x_{\HH}}(\RR^-)}$, we need $\lambda^jr(\lambda^{-1})\in H^1_{\lambda}$ for $-1\leq j \leq 1$ due to $\|\partial_{x_{\HH}}w_{\theta\pm}\|_{L^{\infty}}$.   

For $III'$, we carry out in the same way as done in $III$, and we see that, to estimate $\|\langle x_{\HH}\rangle III'\|_{L^2_{x_{\HH}}(\RR^-)}$, we need $\lambda^jr(\lambda^{-1})\in H^1_{\lambda}$ for $-1\leq j \leq 1$ due to $\|\partial_{x_{\HH}}\mu\|_{L^{\infty}}$.

Putting together all results, the integral \eqref{mu-x-decomp} belongs to $L^2(\langle x_{\HH}\rangle^2dx_{\HH},\RR^-)$ if $\lambda^jr(\lambda^{-1})\in H^1_{\lambda}$ for $-1\leq j \leq 1$.

The inductive argument can be applied to show that $\partial_{x_{\HH}}^{m+1} \int_{\RR}\mu(w_{x-}+w_{x+})ds$ $\in L^2(\langle x_{\HH}\rangle^2dx_{\HH},\RR^-)$  when $\lambda^jr(\lambda^{-1})\in H^1$ for $-1\leq j \leq 1+m$.
\end{proof}
\subsection{$(x_{\HH},t)$-uniform bound on $|\partial_{\HH}m^{(1)}|$}
Lastly, we will show that $|\partial_{x_{\HH}}m^{(1)}_{12}|<1$. This follows easily from the previous proofs with slight modifications. 
\begin{lemma} \label{small-mu}
 Suppose ${\lambda}^{j}r(\lambda^{-1})\in H^1_{\lambda}$ for $-1 \leq j \leq 1$, and the norm
$$\|{\lambda}^{-1}r(1/\lambda)\|_{L_{\lambda}^2}+\|r(1/\lambda)\|_{L^2_{\lambda}}$$ 
is sufficiently small. Then,
$$\|\mu-I\|_{L^2}=o(1),\quad \|\partial_{x_{\HH}}\mu\|_{L^2} = o(1),$$
uniformly in $x_{\HH}$ and $t$.
\end{lemma}
\begin{proof}
From equation \eqref{C-estimate-main}, we have
\begin{equation}\label{C-estimate-1}
\|(C^-_{\RR}w_{\theta +})_{21}\|_{L^2} \leq K_1 \|\lambda^{-1}r(\lambda^{-1})\|_{L^2_{\lambda}},
\end{equation}
and similarly
\begin{equation}\label{C-estimate-2}
\|(C^-_{\RR}\partial_{x_{\HH}}w_{\theta +})_{21}\|_{L^2} \leq K_1 \|r(\lambda^{-1})\|_{L^2_{\lambda}},
\end{equation}
where $K_1$ is some finite constant independent of $x_{\HH}$, $t$, and $r$. We can easily deduce the same bounds for $C^+_{\RR}w_{\theta -}$ as well. From \eqref{mu-ineq1} and \eqref{C-estimate-1}, we have
\begin{equation}\label{mu-uniform-estimate}
\|\mu-I\|_{L^2} \leq K_2  \|\lambda^{-1}r(1/\lambda)\|_{L^2_{\lambda}},
\end{equation}
where $K_2$ is some constant that contains $\|r\|_{L^{\infty}}$ from Proposition \ref{inverse} which is bounded since $r(1/\lambda)\in H^1_{\lambda}$.
From \eqref{mu-ineq2}, using the above estimate, we have
\begin{align}
\|\partial_{x_{\HH}}\mu\|_{L^2} &\leq K_2K_3 \|\mu-I\|_{L^2}+K_2\|r(1/\lambda)\|_{L^2_{\lambda}} \nonumber \\
&\leq K_2^2K_3 \|\lambda^{-1}r(1/\lambda)\|_{L^2_{\lambda}}+K_2\|r(1/\lambda)\|_{L^2_{\lambda}}. \label{mu-x-uniform-estimate}
\end{align}
where $K_3$ is some constant that contains $\|z^{-1}r\|_{L^{\infty}}$ which is bounded since $\lambda r(1/\lambda)\in H^1_{\lambda}$.
\end{proof}

From \eqref{Reconstruct-formula-inverse}, 
we have the $(1,2)$ th element of $\partial_{x_{\HH}}m^{(1)}$ as
$$\partial_{x_{\HH}}m^{(1)}_{12}=\frac{-1}{2\pi i} \partial_x\int_{\RR}[\mu(w_{\theta-}+w_{\theta+})]_{12}ds$$

\begin{lemma} \label{smallness-m}
Suppose the same assumption for $r$ in Lemma \ref{small-mu}, then 
$$|\partial_{x_{\HH}}m^{(1)}_{12}|<1$$
uniformly in $x_{\HH}$ and $t$.
\end{lemma}
\begin{proof}
We want to estimate \eqref{mu-x-decomp} uniformly in $x_{\HH}$ and $t$. 

First, from \eqref{I-Fourier}, we easily see that 
$$|I| \lesssim \|\frac{1}{\lambda}r(1/\lambda)\|_{L^1_{\lambda}}\lesssim \|\frac{1}{\lambda}r(1/\lambda)\|_{L^2_{\lambda}(\langle \lambda \rangle^2d\lambda)}.$$ 
The right-hand side above is bounded since $\|\frac{1}{\lambda}r(1/\lambda)\|_{L^2_{\lambda}}+\|r(1/\lambda)\|_{L^2_{\lambda}}$ is bounded. 

For $II$, from \eqref{II-trick}, \eqref{C-estimate-1} and \eqref{C-estimate-2}, we estimate as follows
\begin{align*}
|II| &\leq \|(C^+w_{\theta-})\|_{L^2}\|(C^-\partial_{x_{\HH}}w_{\theta+})\|_{L^2} + \|(C^-w_{\theta+})\|_{L^2}\|(C^+\partial_{x_{\HH}}w_{\theta-} )\|_{L^2}\\
&\lesssim  \|\lambda^{-1}r(\lambda^{-1})\|_{L^2_{\lambda}} \| r(\lambda^{-1})\|_{L^2_{\lambda}}.
 \end{align*}
The term $I'$ in \eqref{second-term} is estimated in an analogous way. 

Similarly, for $III$, $II'$ and $III'$, from \eqref{C-estimate-1}, \eqref{C-estimate-2}, \eqref{mu-uniform-estimate}, \eqref{mu-x-uniform-estimate}, they can be all estimated as
$$
|III|, |II'| \lesssim \|\lambda^{-1}r(\lambda^{-1})\|_{L^2_{\lambda}}^2 
$$
and 
$$
|III'| \lesssim \|\lambda^{-1}r(\lambda^{-1})\|_{L^2_{\lambda}}(\|\lambda^{-1}r(\lambda^{-1})\|_{L^2_{\lambda}}+\|r(\lambda^{-1})\|_{L^2_{\lambda}}).
$$
In the above inequalities, $\lesssim$ hides some constant dependent on $\|\lambda^jr(1/\lambda)\|_{L^{\infty}_{\lambda}}$, $j=-1,0$.
\end{proof}

\begin{remark}
In the direct problem, Corollary \ref{smallness-a-b} implies that smallness assumption on $r$ in Lemma \ref{small-mu} is satisfied if $\left.q\right|_{t=0}$ is sufficiently small in $X_1$ norm.
\end{remark}

\subsection{Recovery/estimate of $q$ from the reconstruction formula \eqref{Reconstruct-formula-inverse}}
 
\begin{lemma} \label{qH-construct}
Suppose ${\lambda}^{j}r(\lambda^{-1})\in H^1_{\lambda}$ for $-1 \leq j \leq m+1$ and 
$$\|\lambda^{-1}r(1/\lambda)\|_{L^2_{\lambda}}+\|r(1/\lambda)\|_{L^2_{\lambda}}$$
 is sufficiently small. Then, the potential $q_{\HH}(x_{\HH},t)$ of the form 
\begin{equation}\label{reconstruct-formula-2}
q_{\HH}\langle q_{\HH}\rangle^{-1} =\partial_{x_{\HH}}m^{(1)}_{12}
\end{equation}
belongs to $X_m(\RR^-)$ for every $t \geq 0$.
\end{lemma}
\begin{proof}
From Lemma \ref{smallness-m}, we have that
$$\frac{1}{1-|\partial_{x_{\HH}}m^{(1)}_{12}|^2}$$
is bounded. This means that $|q_{\HH}|^2=\frac{|\partial_{x_{\HH}}m^{(1)}_{12}|^2}{1-|\partial_{x_{\HH}}m^{(1)}_{12}|^2}$ is also bounded. Thus, 
\begin{align*}
\|\langle x_{\HH}\rangle q_{\HH}\|_{L^2(\RR^-)} &=\|\langle x_{\HH}\rangle\langle q_{\HH}\rangle \partial_{x_{\HH}}m^{(1)}_{12}\|_{L^2(\RR^-)} \\
&\leq \|\langle q_{\HH}\rangle\|_{L^{\infty}(\RR^-)}\|\langle x_{\HH}\rangle\partial_{x_{\HH}}m^{(1)}_{12}\|_{L^2(\RR^-)} < \infty,
\end{align*}
by Lemma \ref{estimate-formula} if $\lambda^jr(\lambda^{-1}) \in H^1_{\lambda}$ for $-1 \leq j \leq 1$.
In order to estimate $\partial_{x_{\HH}}q_{\HH}$, we note that
\begin{equation} \label{diff}
 \partial_{x_{\HH}}\frac{q_{\HH}}{\langle q_{\HH}\rangle} =\frac{\partial_{x_{\HH}}q_{\HH}}{\langle q_{\HH}\rangle}\left( \frac{\frac{1}{2}|q_{\HH}|^2+1}{\langle q_{\HH}\rangle^2}\right)-\frac{\frac{1}{2}q_{\HH}^2}{\langle q_{\HH}\rangle} \frac{\partial_{x_{\HH}}\bar{q}_{\HH}}{\langle q_{\HH}\rangle^2}
 \end{equation}
and the reverse triangle inequality yields
$$|\partial_{x_{\HH}}\frac{q_{\HH}}{\langle q_{\HH}\rangle}| \geq \frac{|\partial_{x_{\HH}}q_{\HH}|}{\langle q_{\HH}\rangle^3}.$$
What follows is that $|\partial_{x_{\HH}}q_{\HH}| \leq \langle q_{\HH}\rangle^3 |\partial_{x_{\HH}}^2m^{(1)}_{12}|$. This implies that $\langle x_{\HH}\rangle\partial_{x_{\HH}}q_{\HH}\in L^2(\RR^-)$ by Lemma \ref{estimate-formula} if $\lambda^jr(\lambda^{-1}) \in H^1_{\lambda}$ for $-1 \leq j \leq 2$.

We can do the same trick for the higher derivatives, i.e., differentiation of $\frac{q_{\HH}}{\langle q_{\HH} \rangle}$ in $x_{\HH}$ after $m$ times has form 
$$ \partial_{x_{\HH}}^{m}m^{(1)}_{12}= $$
$$\partial_{x_{\HH}}^{m}\frac{q_{\HH}}{\langle q_{\HH}\rangle} =A(\partial^{m-1}_{x_{\HH}}q_{\HH},  \partial^{m-2}_{x_{\HH}}q_{\HH}, \cdots,q_{\HH})+\frac{\partial_{x_{\HH}}^{m}q_{\HH}}{\langle q_{\HH}\rangle}\left( \frac{\frac{1}{2}|q_{\HH}|^2+1}{\langle q_{\HH}\rangle^2}\right)-\frac{\frac{1}{2}q_{\HH}^2}{\langle q_{\HH}\rangle} \frac{\partial_{x_{\HH}}^{m}\bar{q}_{\HH}}{\langle q_{\HH}\rangle^2},$$
where $A$ contains the lower derivatives. Since the lower derivative terms are estimated already, i.e., $\langle x_{\HH}\rangle A \in L^2(\RR^-)$, moving $A$ to the left side and applying the reverse triangle inequality, we have 
$$|\partial_{x_{\HH}}^{m}m^{(1)}_{12}| +|A| \geq \frac{|\partial_{x_{\HH}}^mq_{\HH}|}{\langle q_{\HH}\rangle^3}.$$
This gives $\langle x_{\HH}\rangle \partial_{x_{\HH}}^mq_{\HH}\in L^2(\RR^-)$ by Lemma \ref{estimate-formula} if $\lambda^jr(\lambda^{-1}) \in H^1$ for $-1 \leq j \leq 1+m$.
\end{proof}

In Lemma \ref{qH-construct}, $q_{\HH}$ is estimated in the half line $x_{\HH}\in \RR^-$. In order to address the case of $\RR^+$, we use the different factorization from \eqref{v-triangle}. 
Set 
$$\delta(z)=\exp\left(\frac{1}{2\pi i}\int_{\RR}\frac{\log(1+|r(s)|^2)}{s-z}ds\right)$$
that satisfies $\delta_+=\delta_-(1+|r|^2)$ on $z\in \RR$ and $\delta(z)\rightarrow 1$ as $|z|\rightarrow \infty$, where $\delta_{\pm}$ are the non-tangential limits of $\delta$ to the real line. One can verify that
$$\delta_-^{\sigma_3}(e^{-i \theta \sigma_3}ve^{i \theta \sigma_3})\delta_+^{-\sigma_3}=
\begin{pmatrix}1 & 0 \\ r\delta_-^{-1}\delta_+^{-1}e^{i\theta} & 1\end{pmatrix}\begin{pmatrix} 1 & \bar{r}\delta_-\delta_+e^{-i\theta} \\ 0 & 1 \end{pmatrix}$$
and we set
\begin{equation} \label{second-factorization}
v_-= \begin{pmatrix}1 & 0 \\ - \rho e^{i\theta} & 1\end{pmatrix}, \quad v_+=\begin{pmatrix} 1 & \bar{\rho}e^{-i\theta} \\ 0 & 1 \end{pmatrix},
\end{equation}
$$\rho=r \delta_-^{-1}\delta_+^{-1}.$$
In the following, we denote $\Delta=\delta_-^{-1}\delta_+^{-1}$ and see that $\bar{\rho}=\bar{r}\bar{\Delta}$, since $\overline{\delta_+ \delta_-}=(\delta_+ \delta_-)^{-1}$ that follows from $\delta_+\delta_-=\exp\left(-iH\log(1+|r|^2)\right),$ where $H$ is the Hilbert transform, $Hf=\frac{1}{\pi}\lim_{\epsilon\rightarrow 0}\int_{|z-s|>\epsilon}\frac{f(s)}{s-z}ds$. We have the following estimate
$$\|\lambda \partial_{\lambda} (\delta_-\delta_+)\|_{L^2_{\lambda}}=\|\partial_z(\delta_-\delta_+)\|_{L^2_{z}}=\left\|\frac{\partial_z(|r|^2)}{1+|r|^2}\right\|_{L^2_{z}}\leq \|\lambda \partial_{\lambda}(|r(\lambda^{-1})|^2)\|_{L^2_{\lambda}}.$$
The right-hand side above is bounded if $r(1/\lambda), \lambda r(1/\lambda)\in H^1_{\lambda}$. Therefore, writing
$$\partial_{\lambda}(\lambda^k \rho)=\lambda \partial_{\lambda}(\Delta) \lambda^{k-1}r(1/\lambda)+\Delta\partial_{\lambda}(\lambda^kr(1/\lambda)), \quad k\geq 0,$$
we see that the first and second terms in the right-hand side above are bounded in $L^2_{\lambda}$ if $\lambda^jr(1/\lambda) \in H^1_{\lambda}$ for $j=-1,0,\dots, k$ $(k\geq 1)$, or $j=-1,0,1$ $(k=0)$. Furthermore, $\Delta(1/\lambda)$ at $\lambda=0$ is continuous, since if $\lambda^{-1}r(1/\lambda) \in H^1_{\lambda}$,
$$\frac{1}{\lambda^2}\log(1+|r(1/\lambda)|^2)=\frac{1}{\lambda^2}|r|^2+\mathcal{O}(|r|^4/\lambda^2)$$
for $|\lambda|$ sufficiently small and, thus, $\left.\Delta(1/\lambda)\right|_{\lambda=0}=1$ by the change of variables in the Hilbert transform. 

The above discussion implies that if $\lambda^{j}r(1/\lambda)\in H^1_{\lambda}$ for $j=-1,0,\cdots, m$ with $m\geq 1$, then $\lambda^{j}\rho(1/\lambda)\in H^1_{\lambda}$ for $j=-1,0,\cdots, m$ with $m\geq 1$.

Repetitions of all previous Proofs with the factorization \eqref{second-factorization} can show that Lemma \ref{qH-construct} is extended to estimating $q_{\HH}(x_{\HH})$ for $x_{\HH}$ for the whole line.  

\begin{corollary}[Corollary of Lemma \ref{epsilon-lemma} and Lemma \ref{qH-construct}] \label{x-construct}
Suppose ${\lambda}^{j}r(\lambda^{-1})\in H^1_{\lambda}$ for $-1 \leq j \leq m+1$ with $m\geq 1$ and 
$$\|\lambda^{-1}r(1/\lambda)\|_{L^2_{\lambda}}+\|r(1/\lambda)\|_{L^2_{\lambda}}$$
 is sufficiently small. Let $q_{\HH} \in X_m$ be determined in Lemma \ref{qH-construct}. Define $\epsilon(x,t)$ by
\begin{equation} \label{ode-integral-again}
\epsilon(x,t)=\int_{-\infty}^x(\langle q_{\HH}(y+\epsilon(y,t),t)\rangle-1)dy.
\end{equation}
Then, for every $t\in [0,\infty)$, there exists a unique monotone increasing bounded continuous solution $\epsilon(\cdot,t)$ in \eqref{ode-integral-again} with $\epsilon(-\infty,t)=0$. Thus, the variable $x$ is uniquely defined from
$$x_{\HH}=x+\epsilon(x,t)$$
for every $t\in [0,\infty)$.
\end{corollary}

\begin{theorem} \label{construct-q}
Suppose ${\lambda}^{j}r(\lambda^{-1})\in H^1_{\lambda}$ for $-1 \leq j \leq m+1$ with $m\geq 1$ and 
$$\|\lambda^{-1}r(1/\lambda)\|_{L^2_{\lambda}}+\|r(1/\lambda)\|_{L^2_{\lambda}}$$
 is sufficiently small. Let $q_{\HH} \in X_m(dx_{\HH})$ be determined in Lemma \ref{qH-construct}, and let $\epsilon(x,t)$ and $x$ be determined in Corollary \ref{x-construct}. Then,
$$q(x,t)=q_{\HH}\left(x+\epsilon(x,t),t\right)$$
belongs to $X_m(\RR)$ for every $t \geq 0$.
\end{theorem}
\begin{proof}
It is fruitful to recall the change of variables, 
$$\int_{\RR}q_{\HH}(x_{\HH},t)dx_{\HH}=\int_{\RR}q_{\HH}(x+\epsilon(x,t),t)\langle q_{\HH}(x+\epsilon(x,t),t)\rangle dx=\int_{\RR} q(x,t)\langle q(x,t)\rangle dx.$$
We have
$$\|q_{\HH}\|_{L^2_{x_{\HH}}}^2=\int_{\RR}|q_{\HH}(x_{\HH},t)|^2dx_{\HH}=\int_{\RR}|q(x,t)|^2\langle q(x,t) \rangle dx \geq \|q\|_{L^2_x}^2,$$
that is, if $q_{\HH}\in L^2_{x_{\HH}}$, then $q \in L^2_x$.

Since $\langle A\rangle \langle B\rangle \geq \langle A+B\rangle$ for $A,B \in \RR$, we have the estimate
$$\langle x_{\HH}\rangle \langle \int_{\RR}\HH dy\rangle \geq \langle x\rangle.$$ 
The weight $\langle x_{\HH}\rangle$ bounds $\langle x\rangle$, as $\langle \int_{\RR}\HH dy\rangle$ is bounded if $q_{\HH}\in X_0(dx_{\HH})$. Then, $q_{\HH}\in L^2_{x_{\HH}}(\langle x_{\HH}\rangle^2dx_{\HH})$ implies $q\in L^2_x(\langle x\rangle^2dx)$.

Since $|q_{\HH}|$ is bounded from Lemma \ref{smallness-m}. Since $\langle q_{\HH}\rangle \partial_{x_{\HH}}=\partial_x$, we have  
$$\|\partial_xq\|_{L^2_x}\leq \|\langle q_{\HH}\rangle \partial_{x_{\HH}}q_{\HH}\|_{L^2_{x_{\HH}}}<\infty.$$
More generally, we have
$$\|\langle x\rangle\partial_x^mq\|_{L^2_x}\lesssim \|\langle q_{\HH}\rangle^m \langle x_{\HH}\rangle\partial_{x_{\HH}}^mq_{\HH}\|_{L^2_{x_{\HH}}}$$
$$\leq \|\langle q_{\HH}\rangle^m\|_{L^{\infty}_{x_{\HH}}} \| \langle x_{\HH}\rangle\partial_{x_{\HH}}^mq_{\HH}\|_{L^2_{x_{\HH}}}<\infty,$$
if $q_{\HH} \in X_m(dx_{\HH})$.
\end{proof}

As summary, we have the following diagram of the scattering and inverse scattering maps, i.e., for $n\geq 2$,

\begin{tikzpicture}[decoration={markings,
mark=at position 0.5 with {\arrow[line width=1pt]{>}}}
]

\path[draw, postaction=decorate] (-3,-1)  node[left]{$X_{n+1}\ni q$}-- (3,-1) node[right]{$\lambda^jr(1/\lambda) \in H^1_{\lambda},\; -1\leq j \leq n$};
\path[draw, postaction=decorate] (3,-1) -- (-3,-2) node[left]{$X_{n-1}\ni q$};
\end{tikzpicture}

The forward arrow follows from Theorem \ref{b-a-regularity}. The inverse arrow follows from Theorem \ref{construct-q}, where we set $m=n-1$.

Suppose that $q_0\in X_{\infty}$ and $\|q_0\|_{X_1}$ is sufficiently small. Let $q(x,t) \in X_{\infty}$ be a solution of the WKI equation for $t\in (0,T)$ for some $T>0$ with an initial data $q(\cdot,0)=q_0$. Theorem \ref{b-a-regularity} implies that the reflection coefficient $r$ satisfies $\lambda^{j}r(1/\lambda)\in H_{\lambda}^1$ for all $j\in \mathbb{N}\cup\{-1\}$. Theorem \ref{construct-q} implies that $q(\cdot,t)\in X_m$ for every $t\geq 0$ for all $m \in \mathbb{N}$. This implies that $T$ is an arbitrary large number. This proves Theorem \ref{main-theorem}.

\section{Soliton Solution} \label{soliton}
In the scattering problem, we have constructed the sectionally analytic matrix function $m(z)$ in \eqref{normalized-m}. We shall simply express components of them as 
$$f^+:= (m^{(+)})_1e^{-\frac{i}{z}\int_x^{\infty}\HH dy}, \quad g^+:= (m^{(+)})_2e^{-\frac{i}{z}\int_{-\infty}^x\HH dy},$$
$$f^-:= (m^{(-)})_1e^{\frac{i}{z}\int_{-\infty}^x\HH dy}, \quad g^-:= (m^{(-)})_2e^{\frac{i}{z}\int_x^{\infty}\HH dy}.$$
and 
$$c(z):=a(-1/z)e^{-\frac{i}{z}\int_x^{\infty}\HH dy}.$$
We simply have the expression
$$m(z) = \left\{ \begin{matrix} \left[ \frac{1}{c(z)}f^+, g^+ \right] & z \in \CC^+ \\
 [f^-,  \frac{1}{\overline{c(\bar{z})}}g^-] & z \in \CC^- 
 \end{matrix} \right.
$$
For simplicity, we derive one soliton solution. The method is easily generalized to the case of $N$-soliton solution. For one soliton case, we suppose that $r(z)=0$, that is, 
\begin{equation}\label{r-zero}
m_+=m_-, \quad z\in \RR
\end{equation}
and $c(z)$ has a simple zero at $z_1$ bounded away from $\RR$. The goal is to recover the potential $q$ that corresponds to the above assumptions.
Since $c(z_1)=0$, then so is $a(-1/z_1)=0$. From $a=\det((\psi_+)_1,(\psi_-)_2)$ in \eqref{a-det}, we see that $(\psi_+)_1$ and $(\psi_-)_2$ are linearly dependent at $\lambda_1=-\frac{1}{z_1}$. One can easily verify that this linear dependence implies 
\begin{equation} \label{relation-1}
f^+(z_1)=\gamma e^{2i\theta(z_1)}g^+(z_1),
\end{equation}
where $\gamma\in \CC$ is some constant and $\theta(z_1)=\frac{1}{z_1}x_{\HH}+\frac{2}{z_1^2}t$. Furthermore, by symmetry of the WKI spectral problem, one can verify that $g^-(\bar{z}_1)=\begin{pmatrix} 0 & -1\\ 1 &0\end{pmatrix} \overline{f^+(z_1)}$ and $f^-(\bar{z}_1)=\begin{pmatrix} 0 & 1\\ -1 &0\end{pmatrix} \overline{g^+(z_1)}$. Applying these relations to \eqref{relation-1}, we have
\begin{equation} \label{relation-2}
g^-(\bar{z}_1)=-\bar{\gamma} e^{-2i\theta(\bar{z}_1)}f^-(\bar{z}_1).
\end{equation}
Let the contour $C^+ \subset \CC^+$ be union of a semi-circle and a small circle that enclose $z_1$ (for the case of $N$-solitons, there are $N$ small circles) as shown below
\begin{figure}[htbp] 
   \centering
 \begin{tikzpicture}[decoration={markings,
mark=at position 0.5 with {\arrow[line width=1pt]{>}},
}
]

\path[draw,postaction=decorate] (-3,0) -- (3,0);
\path[draw, postaction=decorate] (3,0) arc(0:180:3);

\path[draw, postaction=decorate] (2,1) arc(360:0:0.2);

\path[draw] (1.8,0) -- (1.8,0.8) [dashed];

\draw[fill] (1.8,1) circle [radius = 0.02];


\node at (0,2) {$D^+$};
\end{tikzpicture}
\end{figure}

Let $C^-\subset \CC^-$ be the contour that is the reflected contour of $C^+$ with respect to the $x$ axis after reversing orientations. Small clock-wise circle of $C^-$ encloses $\bar{z}_1$. By the Cauchy integral formula, we have
$$\frac{1}{2\pi i}\int_{C^{\pm}}\frac{m(s)-I}{s-z}ds=\left\{ \begin{matrix} m(z)-I & z \in D^{\pm}
\\ 0 & z \in \CC^{\mp} \end{matrix}\right.
$$
We can express 
\begin{equation} \label{cauchy-integral}
m(z)-I=\frac{1}{2\pi i}\left(\int_{C^{+}}+\int_{C^{-}}\right)\frac{m(s)-I}{s-z}ds, \quad z\in D^+\cup D^-.
\end{equation}
In the above equation \eqref{cauchy-integral}, we carry out the following procedures--take the radius of the semi-circle in $C^{\pm}$ to infinity and of the small circle in $C^{\pm}$ to zero, the horizontal line in $C^{\pm}$ to the real line, and apply relations \eqref{r-zero}, \eqref{relation-1}, and \eqref{relation-2}. The resulting equation of equation \eqref{cauchy-integral} is  

$$m(z)=I+\frac{\bar{\beta} e^{-2i\theta(\bar{z}_1)}}{\bar{z}_1-z}m(\bar{z}_1)\begin{pmatrix} 0 & 1 \\ 0 & 0\end{pmatrix}-\frac{\beta e^{2i\theta(z_1)}}{z_1-z}m(z_1)\begin{pmatrix} 0 & 0 \\ 1 & 0\end{pmatrix},$$
where $\beta=\gamma/a'(z_1)$. 
From above, we can solve for $g^+(z_1)$ and $f^-(\bar{z}_1)$. In particular, the first components of $f^-(\bar{z}_1)$ and $g^+(z_1)$ are found as  
$$(f^-(\bar{z}_1))_1=\frac{1}{1+\frac{|\beta|^2}{|z_1-\bar{z}_1|^2}e^{2i\theta(z_1)-2i\theta(\bar{z}_1)}},\quad (g^+(\bar{z}_1))_1=\frac{\bar{\beta}}{\bar{z}_1-z_1}\frac{e^{-2i\theta(\bar{z}_1)}}{1+\frac{|\beta|^2}{|z_1-\bar{z}_1|^2}e^{2i\theta(z_1)-2i\theta(\bar{z}_1)}}.$$
Choose $\frac{|\beta|}{|z_1-\bar{z}_1|}=1$. Also, $\mbox{arg}\beta=0$ for convenience. Also, denote $\lambda_1=\xi+i\eta$ for $z_1=-1/\lambda_1$. The $(1,2)$ th element of $\lim_{z\rightarrow \infty}zm(z)$ is given by 
\begin{align}
m^{(1)}_{12} &=\lim_{z\rightarrow \infty}z(m(z))_{12} \nonumber\\
&=-\bar{\beta}e^{-\theta(\bar{z}_1)}(f^-(\bar{z}_1))_1\nonumber\\
&=-\frac{2\eta}{\xi^2+\eta^2}\frac{1}{e^{-2i\bar{\lambda}_1x_{\HH}+4i\bar{\lambda}_1^2t}+e^{-2i \lambda_1x_{\HH}+4i \lambda_1^2t}}\nonumber \\
&=-\frac{2\eta}{\xi^2+\eta^2}\frac{e^{-i(-2\xi x_{\HH}+4(\xi^2-\eta^2))t}}{e^{-2\eta x_{\HH}+8\xi\eta t}+e^{2\eta x_{\HH}-8\xi\eta t}}. \label{m12}
\end{align}
The $(1,1)$ th element of $\lim_{z\rightarrow \infty}z(m(z)-I)$ is given by
\begin{align}
m^{(1)}_{11} &=\lim_{z\rightarrow \infty}z(m(z)-I)_{11} \nonumber\\
&=-\beta e^{-\theta(z_1)}(g^+(\bar{z}_1))_1\nonumber\\
&=i\frac{2\eta}{\xi^2+\eta^2}\frac{e^{2\eta x_{\HH}-8\xi\eta t}}{e^{-2\eta x_{\HH}+8\xi\eta t}+e^{2\eta x_{\HH}-8\xi\eta t}}. \label{m11}
\end{align}

Differentiating \eqref{m12} in $x_{\HH}$ gives 
$$\partial_{x_{\HH}}m^{(1)}_{12}=2 \frac{2\eta}{\xi^2+\eta^2}e^{-i(-2\xi x_{\HH}+4(\xi^2-\eta^2))t} [\eta \Tanh(*)-i\xi]\sech(*),$$
where $*=2\eta x_{\HH}-8\xi\eta t$. From \eqref{reconstruct-potential}, the potential $q_{\HH}$ is related to $\partial_{x_{\HH}}(m_1)_{12}$ as 
$$q_{\HH}\langle q_{\HH}\rangle^{-1} =\partial_{x_{\HH}}m^{(1)}_{12}.$$ 
Since we have
$$|\partial_{x_{\HH}}m^{(1)}_{12}|^2=\frac{2^2\eta^2}{(\xi^2+\eta^2)^2}[\eta^2\Tanh^2(*)+\xi^2]\sech^2(*),$$
we find
$$|q_{\HH}|^2=\frac{4\eta^2}{\xi^2+\eta^2}\frac{\cosh^2(*)-\frac{\eta^2}{\xi^2+\eta^2}}{(\cosh^2(*)-\frac{2\eta^2}{\xi^2+\eta^2})^2},\quad\langle q_{\HH}\rangle=\frac{\cosh(*)^2}{\cosh^2(*)-\frac{2\eta^2}{\xi^2+\eta^2}}.$$
Expressing $\partial_{x_{\HH}}m^{(1)}_{12}$ as 
$$\partial_{x_{\HH}}m^{(1)}_{12}=2i \frac{2\eta}{\sqrt{\xi^2+\eta^2}}e^{-i(-2\xi x_{\HH}+4(\xi^2-\eta^2))t} \cosh(2\eta x_{\HH}-8\xi\eta t+i\alpha)\sech^2(*)$$
where $\frac{\eta}{\xi}=\tan\alpha$, we obtain
$$q_{x_{\HH}}=\langle q_{x_{\HH}}\rangle \partial_{x_{\HH}}m^{(1)}_{12}=2i\frac{\eta}{\sqrt{\xi^2+\eta^2}}e^{-i(-2\xi x_{\HH}+4(\xi^2-\eta^2))t} \frac{ \cosh(2\eta x_{\HH}-8\xi\eta t+i\alpha)}{\cosh^2(2\eta x_{\HH}-8\xi\eta t)-\frac{2\eta^2}{\xi^2+\eta^2}}.$$
Lastly, recall $x_{\HH}=x+\epsilon(x,t)$ and from the formula \eqref{x-explicit}, we can recover $\epsilon$ as
$$\epsilon(x,t)=i^{-1} \lim_{z\rightarrow \infty}z(m(z)-I)_{11}=i^{-1} m^{(1)}_{11}.$$
We have done this computation in \eqref{m11}, so we obtain the expression
\begin{equation}\label{epsilon} 
\epsilon(x,t)=\frac{\eta}{\xi^2+\eta^2}[\tanh(2\eta(x-4\xi t+\epsilon(x,t)))+1].
\end{equation}
We finally arrived the one soliton solution
$$q(x,t)=2i\frac{\eta}{\sqrt{\xi^2+\eta^2}}e^{-i(-2\xi (x+\epsilon(x,t))+4(\xi^2-\eta^2)t)} \frac{ \cosh(2\eta (x+\epsilon(x,t))-8\xi\eta t+i\alpha)}{\cosh^2(2\eta (x+\epsilon(x,t))-8\xi\eta t)-\frac{2\eta^2}{\xi^2+\eta^2}},$$
where $\epsilon(x,t)$ is found by an implicit equation \eqref{epsilon}. This coincides with one soliton in the original paper \cite{Shimizu-Wadati-1980} after the change of parameter $\lambda \rightarrow -\lambda$ because of our Lax pair with opposite sign for $\lambda$.

\begin{figure}[htbp] 
   \centering
   \includegraphics[width=3.5in]{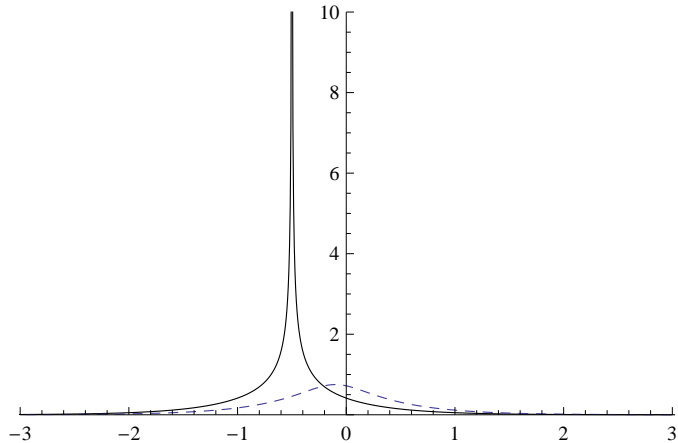} 
   \caption{$|q(x,0)|$ v.s. $x$ :the solid line corresponds to the case $\xi=\eta=1$ and the dashed line corresponds to $\xi=3$ and $\eta=1$.}
   \label{fig:example}
\end{figure}

The case $|\xi|>|\eta|>0$ gives smooth soliton, and the case $|\xi|=|\eta|>0$ gives bursting soliton, whose maximum hight is infinity.


\begin{thebibliography}{10}

\bibitem{Ablowitz-Prinari-Trubatch-2004} M. J. Ablowitz, B. Prinari, and A. D. Trubatch, \emph{Discrete and Continuous Nonlinear Schr\"{o}dinger systems}, Cambridge University Press, Cambridge, 2004.

\bibitem{Boiti-Gerdjikov-Pempinelli-1986} M. Boiti, V. S. Gerdjikov, and F. Pempinelli,
\emph{The WKIS  System: B\"{a}cklund Transformations, Generalized Fourier Transforms and All That},
Prog. Theor. Phys., {\bf 75} (1986), 1111--1140.

\bibitem{Boutet-Kostenko-Shepelsky-Teschl-2009} A. Boutet De Monvel, A. Kostenko, D. Shepelsky, and G. Teschl, 
\emph{Long-time asymptotics for the Camassa-Holm equation},
SIAM J. Math. Anal., {\bf 41}(2009), 1559--1588.

\bibitem{Constantin-Escher-1998} A. Constantin and J. Escher, 
\emph{Global existence and blow-up for a shallow water equation},
Annali Sc. Norm. Sup. Pisa {\bf 26} (1998), 303--328.

\bibitem{Deift-Zhou-1991} P. Deift and X. Zhou,
\emph{Direct and Inverse Scattering on the Line with Arbitrary Singularities}, 
Comm. Pure Appl. Math., {\bf 44} (1991), 485--533.

\bibitem{Deift-Zhou-2002} P. Deift and X. Zhou,
\emph{Long-time asymptotics for solutions of the NLS equation with initial data in a weighted Sobolev space},
Comm. Pure Appl. Math., {\bf 56} (2003), 1029--1077.

\bibitem{Feng-Inoguchi-Kajiwara-Maruno-2011} B. Feng, J. Inoguchi, K. Kajiwara, K. Maruno, and Y. Ohta,
\emph{Discrete Integrable Systems and Hodograph Transformations Arising from Motions of Discrete Plane Curves},
J. Phys. A: Mathematical and Theoretical, {\bf 44} (2011), 19pp.

\bibitem{Gui-Liu-2010} G. Gui and Y. Liu, 
\emph{On the global existence and wave-breaking criteria for the two-component Camassa-Holm system},
J. Funct. Anal., {\bf 258} (2010), 4251--4278.

\bibitem{Gui-Liu-Olver-Qu-2013} G. Gui, Y. Liu, P. J. Olver, and C. Qu, 
\emph{Wave-Breaking and Peakons for a Modified Camassa-Holm Equation},
Commun. Math. Phys., {\bf 319} (2013), 731--759.

\bibitem{Ishimori-1982} Y. Ishimori, 
\emph{A Relationship between the Ablowitz-Kaup-Newell-Segur and Wadati-Konno-Ichikawa Schemes of the Inverse Scattering Method},
J. Phys. Soc. Jpn., {\bf 51} (1982), 3036--3041.

\bibitem{Ichikawa-Konno-Wadati-1981} Y. Ichikawa, K. Konno, and M. Wadati,
\emph{Nonlinear Transverse Oscillation of Elastic Beams under Tension},
J. Phys. Soc. Jpn., {\bf 50} (1981), 1799--1802.

\bibitem{Konno-Ichikawa-Wadati-1981-2} K. Konno, Y. Ichikawa, and M. Wadati,
\emph{A Loop Soliton Propagating along a Stretched Rope},
J. Phys. Soc. Jpn., {\bf 50} (1981), 1025--1026.

\bibitem{Liu-Pelinovsky-Sakovich-2009} Y. Liu, D. Pelinovsky, and A. Sakovich, 
\emph{Wave breaking in the short-pulse equation},
Dynam. Part. Differ, Eq., {\bf 6} (2009), 291 -- 310.


\bibitem{Liu-Yin-2006} Y. Liu and Z. Yin,
\emph{Global Existence and Blow-Up Phenomena for the Degasperis-Procesi Equation},
Commun. Math. Phys., {\bf 267} (2006), 801--820.

\bibitem{Pelinovsky-Sakovich-2010} D. Pelinovsky and A. Sakovich, 
\emph{Global Well-Posedness of the Short-Pulse and Sine-Gordon Equations in Energy Space},
Commun. Part. Diff. Eq., {\bf 35} (2010), 613--629.

\bibitem{Shimizu-Wadati-1980} T. Shimizu and M. Wadati, 
\emph{A New Integrable Nonlinear Evolution Equation}, 
Prog. Theor. Phys., {\bf 63} (1980), 808--820.

\bibitem{Wadati-Konno-Ichikawa-1979} M. Wadati, K. Konno, and Y. Ichikawa,
\emph{New Integrable Nonlinear Evolution Equations}, 
J. Phys. Soc. Jpn., {\bf 47} (1979), 1698--1700.

\bibitem{Wadati-Ichikawa-Shimizu-1980} M. Wadati, Y. Ichikawa, and T. Shimizu, 
\emph{Cusp Soliton of a New Integrable Nonlinear Evolution Equation},
Prog. Theor. Phys., {\bf 64} (1980), 1959--1967.

\bibitem{Wadati-Sogo-1983} M. Wadati and M. Sogo,
\emph{Gauge Transformations in Soliton Theory},
J. Phys. Soc. Jpn. {\bf 52} (1983), 394--338. 

\bibitem{Xu-2016} Jian Xu, 
\emph{Long-time asymptotics for the Short-Pulse equation},
arXiv:1608.03057v1, Aug 2016.

\bibitem{Zhou-1989} Xin Zhou,
\emph{The Riemann-Hilbert Problem and Inverse Scattering},
SIAM J. Math. Anal., {\bf 20} (1989), 966--986.

\bibitem{Zhou-1998} Xin Zhou,
\emph{$L^2$- Sobolev Space Bijectivity of the Scattering and Inverse Scattering Transforms},
Comm. Pure Appl. Math., {\bf 51} (1998), 697--731.

\end{thebibliography}
\end{document}